\documentclass[preprint,12pt]{elsarticle}
\usepackage{}
\usepackage{amsfonts}
\usepackage{mathrsfs}
\usepackage{graphicx,amsmath,amsfonts,amssymb,amscd,amsthm, mathrsfs}
\newtheoremstyle{kai}
{3pt}{3pt}{}{}{\bfseries}{.}{.5em}{}
\makeatletter
\def\EquationsBySection{\def\theequation
{\thesection.\arabic{equation}}%
\@addtoreset{equation}{section}}

\newcommand\old[1]{}
\newcommand{\pend}{\hfill \thicklines \framebox(6.6,6.6)[l]{}}
\renewenvironment{proof}{\noindent {\it  Proof.} \rm}{\pend}
\newtheorem{theorem}{Theorem}[section]
\newtheorem{lemma}{Lemma}[section]
\newtheorem{corollary}{Corollary}[section]
\newtheorem{proposition}{Proposition}[section]

\newtheorem{definition}{Definition}[section]

\journal{}
\EquationsBySection
\makeatother

\begin{document}

\begin{frontmatter}

\title{\bf\Large Stationary Distributions of Second Order Stochastic Evolution Equations with Memory in\\ Hilbert Spaces}


\author{
{\bf Kai Liu}}
\address{
 Department of Mathematical Sciences,\\ School of Physical Sciences,\\
The University of Liverpool,\\ Liverpool, L69 7ZL, U.K.\\ E-mail: k.liu@liverpool.ac.uk}

\begin{abstract}

 In this paper, we consider stationarity of a class of second-order stochastic evolution equations with memory, driven by Wiener processes or L\'evy jump processes, in Hilbert spaces. The strategy is to formulate by reduction some first-order systems in connection with the stochastic equations under investigation. We develop asymptotic behavior of dissipative second-order equations and then apply them to time delay systems through Gearhart-Pr\"uss-Greiner's theorem. The stationary distribution of the system under consideration is the projection on the first coordinate of the corresponding stationary results of a lift-up stochastic system without delay on some product Hilbert space. Last, an example of stochastic damped delay wave equations with memory is presented to illustrate our theory.  \\ \\
\end{abstract}
\end{frontmatter}


\noindent {\bf 2000 Mathematics Subject Classification(s):} 60H15, 60G15, 60H05.



\newpage
\section{\large Introduction}

For any Hilbert spaces $H$ and $K$, we denote by ${\mathscr L}(H, K)$ the space of all bounded linear operators from $H$ into $K$. If $H=K$, we simply write ${\mathscr L}(H)$ for ${\mathscr L}(H, H)$.
Let $r>0$ and consider the following second-order stochastic abstract Cauchy problem with memory on a Hilbert space ${H}$,
\begin{equation}
\label{12/09/17(13489)}
\begin{cases}
d\Big(\displaystyle\frac{du(t)}{dt}\Big)  + Au(t)dt = Bu'(t)dt+ Mu_tdt + Nu'_tdt  + R(u(t), u'(t), u_t, u'_t)dZ(t),\hskip 15pt t\ge 0,\\
u(0)=\phi_{0, 1},\,\,\,u'(0)=\phi_{0, 2},\,\,\,u_0=\phi_{1,1},\,\,\,u'_0 =\phi_{1, 2},
\end{cases}
\end{equation}
where $\phi_{i,j}$, $i=0,1,\,j=1,2$, are appropriate initial data and $Z$ could be an abstract $Q$-Wiener process or L\'evy jump process. Here $u_t(\theta) := u(t+\theta)$, $u'_t(\theta) := u'(t+\theta)$, $\theta\in [-r, 0]$ and $(A, {\mathscr D}(A))$, $(B, {\mathscr D}(B))$ are two linear operators from $H$ into itself,  $M,\,N$ are linear mappings from $L^2([-r, 0], H)$ into $H$, respectively, and $R$ is a measurable nonlinear mapping from $ H\times H\times L^2([-r, 0], H)\times L^2([-r, 0], H)$ to some space of linear operators (see Sections 4, 5 and 6 for precise definitions). In this work, we intend to consider stationarity of solutions for a class of stochastic second-order evolution equations with memory, i.e.,  (\ref{12/09/17(13489)}).

We remark that there exists an exhaustive literature on deterministic second-order abstract Cauchy problems, e.g., see Fattorini \cite{faho85} and references therein.
 There are some works devoted to this type of equations with memory features, and we refer, for instance, to \cite{kjern00} for some fundamental statements and to \cite{absp05} for some recent results
  about this topic. On the other hand, it is worth pointing out that stochastic abstract  second-order Cauchy problems with memory such as (\ref{12/09/17(13489)}) have not been investigated in depth until now. The only works known to the current author in the existing literature are \cite{flhjg2014, flzhg2014, yrds09, rysr12} in which the well-posedness and stability of stochastic systems with an infinite delay driven by a jump process and a neutral or impulsive term are considered.

  The whole organisation of this work is as follows. After having developed some necessary tools to study the well-posedness of dissipative linear wave equations in Section 2, we shall consider the asymptotic behavior of abstract deterministic second-order evolution equations in Section 3 and some generalized second-order evolutions with memory in Section 4. We shall show in Section 5 how to reduce a second-order stochastic equation with memory to an abstract stochastic Cauchy problem on appropriate product Hilbert spaces. Consequently, we can use in Section 5 some recent methods to investigate the stationarity of the stochastic systems under investigation. In Section 6, we carry on to consider the stationarity of stochastic delay  systems driven by a L\'evy jump process. Last, we apply our theory to an illustrative example, i.e., a stochastic damped wave equation with memory in Section 7.

\section{\large Dissipative Wave Equations}

Consider the following second-order abstract Cauchy problem on a Hilbert space ${H}$,
\begin{equation}
\label{12/09/17(1)}
\begin{cases}
u''(t)  + Au(t) = Bu'(t),\hskip 15pt t\ge 0,\\
u(0)=\phi_{0, 1},\,\,\,u'(0)=\phi_{0, 2},
\end{cases}
\end{equation}
where $B: {\mathscr D}(B)\subset H\to H$ is some linear, possibly unbounded, operator and
 $A: {\mathscr D}(A)\subset H\to H$ is a self-adjoint linear operator such that
 \[
 \langle Au, u\rangle_H \ge \alpha \|u\|^2_V,\hskip 20pt \alpha>0,\hskip 20pt \forall\, x\in {\mathscr D}(A),\]
 where $V := {\mathscr D}(A^{1/2})$, equipped with the usual graph norm $\|u\|_V := \|A^{1/2}u\|_H$, $u\in V$.
By using the standard reduction $y(t) := \displaystyle{u(t)\choose u'(t)}$, we intend to transform (\ref{12/09/17(1)}) into a first-order system
\begin{equation}
\label{12/09/17(02)}
\begin{cases}
y'(t) = \Lambda y(t),\hskip 15pt t\ge 0,\\
y(0) =\displaystyle{\phi_{0, 1}\choose \phi_{0, 2}}\in {\mathbb H},
\end{cases}
\end{equation}
for the matrix operator
\[
\Lambda = \left(\begin{array}{cc}
0& I\\
-A& B
\end{array}\right)\]
with domain
\[
{\mathscr D}(\Lambda) = {\mathscr D}(A)\times ({\mathscr D}(A^{1/2})\cap {\mathscr D}(B))\]
in the product Hilbert space ${\mathbb H}:= V\times {H}$, equipped with the usual product space inner product and norm. To this end, we remark that $A^{1/2}: V\to H$ is a unitary operator. This allows us to consider a unitary operator
\[
\Sigma =  \left(\begin{array}{cc}
A^{1/2}& 0\\
0& I
\end{array}\right)\in {\mathscr L}({\mathbb H}, \tilde{\mathbb H}),
\]
where $I$ is the identity operator on $H$ and $\tilde{\mathbb H} := H\times H$, equipped with the usual product space inner product and norm. Further, define a matrix operator on ${\tilde{\mathbb H}}$,
\[
\Lambda_0 = \left(\begin{array}{cc}
0& A^{1/2}\\
-A^{1/2}& B
\end{array}\right)\]
with domain
\[
{\mathscr D}(\Lambda_0) = {\mathscr D}(A^{1/2})\times ({\mathscr D}(A^{1/2})\cap {\mathscr D}(B)).\]
Then we know that $\Sigma$ defines a unitary equivalence between $\Lambda$ and $\Lambda_0$, i.e.,
\[
\Lambda_0 = \Sigma\Lambda \Sigma^{-1}.\]
Hence, in order to deal with the well-posedness of (\ref{12/09/17(02)}) it suffices to consider the same problem of the following equation
\begin{equation}
\label{12/09/17(2)}
\begin{cases}
y'(t) = \Lambda_0 y(t),\hskip 15pt t\ge 0,\\
y(0) =\displaystyle{\phi_{0, 1}\choose \phi_{0, 2}}\in \tilde{\mathbb H},
\end{cases}
\end{equation}
in the product Hilbert space $\tilde{\mathbb H}$.

 To proceed further, we want to find conditions under which the inverse operator $\Lambda^{-1}_0$ of $\Lambda_0$ exists. In fact, if $\Lambda_0$ is invertible, we know by the closed graph theorem that $\Lambda^{-1}_0$ is a bounded linear operator on $\tilde{\mathbb H}$ and we can represent it in a matrix form
\[
\Lambda^{-1}_0 = \left(\begin{array}{cc}
U& V\\
W& S
\end{array}\right)\in {\mathscr L}(\tilde{\mathbb H}).\]
To identify $\Lambda^{-1}_0$ explicitly, we can formally conclude from $\Lambda^{-1}_0\Lambda_0 =I|_{{\mathscr D}(\Lambda_0)}$ that its entries satisfy
\[
\begin{split}
&Ux= {A^{-1/2}}BA^{-1/2}x\hskip 15pt \hbox{for all}\hskip 15pt x\in {\mathscr D}({A^{-1/2}}BA^{-1/2}),\\
&V=-{A^{-1/2}},\,\, W= A^{-1/2}\,\,\,\, \hbox{and}\,\,\,\, S=0.
\end{split}
\]
These operators give rise to a bounded inverse of $\Lambda_0$ if and only if
$A^{1/2}({\mathscr D}(B)\cap {\mathscr D}(A^{1/2}))$ is dense in ${H}$ and
${{A^{-1/2}}BA^{-1/2}}\in {\mathscr L}({H})$.
As a matter of fact, we have the following result.

\begin{theorem}
\label{05/10/2017(1)}
Assume that $B: {\mathscr D}(B)\subset H\to {H}$ is densely defined, closed and dissipative, i.e., $Re\,\langle Bu, u\rangle_{{H}}\le 0$ for all $u\in {\mathscr D}(B)$ and ${\mathscr D}(B)\cap {\mathscr D}(A^{1/2})$ is dense in $H$.
Then operator $\Lambda_0$ is densely defined, dissipative and closed. Further, suppose  that $A^{1/2}({\mathscr D}(B)\cap {\mathscr D}(A^{1/2}))$ is dense in ${H}$ and
${A^{-1/2}BA^{-1/2}}\in {\mathscr L}({H})$,
 then $\Lambda_0$ generates a contraction $C_0$-semigroup $e^{t\Lambda_0}$, $t\ge 0$, on $\tilde{\mathbb H}$, i.e., $\|e^{t\Lambda_0}\|\le 1$ for all $t\ge 0$.
\end{theorem}

\begin{proof} It is immediate that $\Lambda_0$ is densely defined and closed. Now we show the dissipativity of $\Lambda_0$. By assumption, since $B$ is dissipative, we have for any $\displaystyle{u\choose v}\in {\mathscr D}({\Lambda_0})$ that
\[
\begin{split}
Re\Big\langle \Lambda_0{u\choose v}, {u\choose v}\Big\rangle_{\tilde{\mathbb H}} &= Re\Big\langle \left(\begin{array}{cc}
0& A^{1/2}\\
-A^{1/2}& B
\end{array}\right){u\choose v}, {u\choose v}\Big\rangle_{\tilde{\mathbb H}}\\
 & =  Re\langle A^{1/2}v, u\rangle_{H} + Re\langle -A^{1/2}u +Bv, v\rangle_{ H}\\
 &= Re\langle Bv, v\rangle_{H}\le 0,
 \end{split}
 \]
 i.e., $\Lambda_0$ is dissipative.

 On the other hand, by assumption, $\Lambda_0$ has a bounded linear inverse $\Lambda^{-1}_0$, $0\in \rho(\Lambda_0)$ which is an open set in ${\mathbb C}$. Hence, there exists a number $\lambda>0$ such that $\lambda\in \rho(\Lambda_0)$ and ${\mathscr R}(\lambda I -\Lambda_0)={\tilde{\mathbb H}}$. By the well-known Lumer-Phillips Theorem, this implies that $\Lambda_0$ generates a contraction $C_0$-semigroup on $\tilde{\mathbb H}$.
\end{proof}

\begin{corollary}
 Assume that $B: {\mathscr D}(B)\subset H\to {H}$ is closed, dissipative and ${\mathscr D}(A^{1/2})\subset {\mathscr D}(B)$, then $\Lambda_0$ generates a contraction $C_0$-semigroup $e^{t\Lambda_0}$, $t\ge 0$, on $\tilde{\mathbb H}$.
 \end{corollary}
 \begin{proof}
 The proof is straightaway since all the conditions in Theorem \ref{05/10/2017(1)} are satisfied on this occasion.
 \end{proof}

 In the sequel of this work, the conditions in Theorem \ref{05/10/2017(1)} are always assumed to hold.

\section{\large Asymptotic Behavior of Equations}

 In this section, we are concerned with the asymptotic behavior of equation
 (\ref{12/09/17(1)}). We first establish a result for a simple case, i.e., $B=\beta I$, $\beta\in {\mathbb R}$, whose argument is shaped by J. Zabczyk and the current author, to show that the strict dissipativity of $B$ is essential for the exponential stability of (\ref{12/09/17(02)}), i.e.,
 $\|e^{t{\Lambda}}\|\le Me^{-\mu t}$, $t\ge 0$, for some constants $M\ge 1$ and $\mu>0$.

    Suppose that  $B=\beta I$, $\beta\in {\mathbb R}$, in (\ref{12/09/17(1)}). First, note that from the definition of the resolvent sets $\rho(-A)$ and $\rho({\Lambda})$, we have that $\lambda\in \rho({\Lambda})$ if and only if $\lambda(\lambda -
\beta)\in \rho(-A)$. Indeed, $\lambda\in\rho({\Lambda})$ if and only if the following equation system defining resolvent operators
$$
\left(\begin{array}{cc}
\lambda I & -I\\ A & \lambda I -\beta I
\end{array}\right)
\left(\begin{array}{c}
y_1\\ y_2
\end{array}\right) =\left(\begin{array}{c}
z_1\\ z_2
\end{array}\right)
$$
has a unique solution in ${\mathscr D}({\Lambda})$ for any $(z_1, z_2)\in {\mathbb H}$, a situation which is possible if and only if $\lambda(\lambda-\beta)\in \rho(-A)$. Let $\sigma(\Lambda)$ denote the spectral set of ${\Lambda}$ and define the spectral bound $\omega_s({\Lambda})$ of $\Lambda$ by
\[
\omega_s({\Lambda}) = \sup\{Re\,\lambda:\, \lambda\in \sigma({\Lambda})\},\]
then we have
\[
\omega_s({\Lambda}) = \sup\{Re\,\lambda:\, \lambda(\lambda-\beta)\in \sigma({-A})\}.\]
In other words, $\omega_s({\Lambda})$ is the larger real part of the roots to the equation
\begin{equation}
\label{29/09/17(1)}
\lambda^2 -\beta\lambda -\omega_s(-A)=0
\end{equation}
with $\lambda(\lambda-\beta)\in \sigma({-A})$.
Since $-A$ is self-adjoint, $\sigma(-A)$ is a subset of ${\mathbb R}$ and so we have $\lambda(\lambda-\beta)\in {\mathbb R}$. Let $\lambda=a+ib$. Since $(a+ib)(a+ib -\beta)\in {\mathbb R}$, it follows that
\[
b=0\hskip 20pt \hbox{or}\hskip 20pt b\not= 0, \,\,2a-\beta=0.\]
If $b=0$, this means that $\omega_s({\Lambda})$ equals the larger real root of (\ref{29/09/17(1)}), which is
\[
\omega_s({\Lambda}) = \displaystyle{\frac{\beta}{2}}+ \sqrt{{\frac{\beta^2}{4}}+\omega_s(-A)}\hskip 10pt\,\,\hbox{if}\hskip
10pt {\frac{\beta^2}{4}}+\omega_s(-A)\ge 0.\]
If $b\not= 0$ and $2a-\beta=0$, this means that (\ref{29/09/17(1)}) has non real roots in this case and $\omega_s({\Lambda}) = a=\beta/2.$
In other words, we have the following relation of spectrum bounds between $\omega_s(-A)$ and $\omega_s({\Lambda})$
\begin{equation}
\label{12/06/2012(1)}
\omega_s({\Lambda}) =
\begin{cases}
\displaystyle{\frac{\beta}{2}}+ \sqrt{{\frac{\beta^2}{4}}+\omega_s(-A)}\hskip 10pt\,\,\hbox{if}\hskip
10pt {\frac{\beta^2}{4}}+\omega_s(-A)\ge 0,\\
\\
\displaystyle{\frac{\beta}{2}}\hskip 100pt
\,\,\,\,\hbox{otherwise}.
\end{cases}
\end{equation}
Define the growth bound $\omega_g({\Lambda})$ of $\Lambda$ by
 \[
\omega_g(\Lambda) = \inf\Big\{\mu\in {\mathbb R}:\,\, \hbox{there exists}\,\, M\ge 1\,\,\hbox{such that}\,\, \|e^{t{\Lambda}}\|\le M e^{\mu t}\,\,\hbox{for all}\,\, t\ge 0\Big\}.\]
From (\ref{12/06/2012(1)}), we have immediately the following conclusions (i) and (ii).

(i) If $\omega_s(-A)\ge 0$, we have $\omega_g({\Lambda})\ge \omega_s({\Lambda})\ge 0$ from (\ref{12/06/2012(1)})  and system  (\ref{12/09/17(02)})  is thus exponentially unstable for any $\beta\in {\mathbb R}$.

(ii) If $\omega_s(-A)<0$ and $\beta\ge  0$, then (\ref{12/06/2012(1)}) implies that $\omega_g({\Lambda})\ge \omega_s({\Lambda})\ge 0$, and system (\ref{12/09/17(02)}) is thus exponentially unstable.

(iii) If $\omega_s(-A)<0$ and $\beta=-\alpha$, $\alpha>0$, $B=-\alpha I$, we  may show that the operator
\begin{equation}
\label{12/06/2012(50)}
{P} =
\begin{pmatrix}
\displaystyle \frac{1}{\alpha}I + \frac{\alpha}{2}A^{-1}& \displaystyle\frac{1}{2}A^{-1}\\
\\
\displaystyle\frac{1}{2}I &\displaystyle\frac{1}{\alpha}I\\
\end{pmatrix}
\end{equation}
is the unique self-adjoint, non-negative solution of  Lyapunov equation
\begin{equation}
\label{12/06/2012(51)}
\left\langle {\Lambda}
\begin{pmatrix}
y_1\\ y_2\\
\end{pmatrix}, {P}\begin{pmatrix}
y_1\\ y_2\\
\end{pmatrix}\right\rangle_{\mathbb H} + \left\langle
{P}\begin{pmatrix}
y_1\\ y_2\\
\end{pmatrix}, {\Lambda}\begin{pmatrix}
y_1\\ y_2\\
\end{pmatrix}\right\rangle_{\mathbb H} =- \left\|\begin{pmatrix}
y_1\\ y_2\\
\end{pmatrix}\right\|_{\mathbb H}^2
\end{equation}
for any $y_1\in {\mathscr D}(A)$, $y_2\in {\mathscr D}(A^{1/2})$.
In this case, the following estimates hold:
\begin{equation}
\label{12/06/2012(5356)}
\gamma_{-}\left\|
\begin{pmatrix}
y_1\\ y_2\\
\end{pmatrix}\right\|_{\mathbb H}^2 \le
\left\langle {P}
\begin{pmatrix}
y_1\\ y_2\\
\end{pmatrix}, \begin{pmatrix}
y_1\\ y_2\\
\end{pmatrix}\right\rangle_{\mathbb H}\le \gamma_{+} \left\|
\begin{pmatrix}
y_1\\ y_2\\
\end{pmatrix}\right\|^2_{\mathbb H},\hskip 20pt\begin{pmatrix}
y_1\\ y_2\\
\end{pmatrix}\in {\mathbb H},
\end{equation}
where
\[
\gamma_{-} = \frac{1}{\alpha}\cdot\frac{\sqrt{1+\theta}}{1+\sqrt{1+\theta}}>0,\hskip 20pt \gamma_+=\frac{1}{\alpha}\Big(1+ \frac{1 +\sqrt{1+\theta}}{\theta}\Big)>0,\hskip 20pt \theta =\frac{4|\omega_s(-A)|}{\alpha^2}.\]
Moreover, we have that
\[
\|e^{t{\Lambda}}\|\le \sqrt{\frac{\gamma_{+}}{\gamma_-}}e^{-\frac{1}{2\gamma_+}t},\hskip 20pt t\ge 0.\]
That is,  system  (\ref{12/09/17(02)}) in this case is exponentially stable.

Indeed,  it is straightforward to get symmetry and nonnegativity of ${P}$ and (\ref{12/06/2012(51)}) is easily verified by a direct calculation.
To show (\ref{12/06/2012(5356)}),  note that for any $\begin{pmatrix}y_1\\ y_2\\
\end{pmatrix}\in {\mathbb H}$,
\begin{equation}
\label{12/06/2012(53556)}
\begin{split}
\left\langle {P}\begin{pmatrix}
y_1\\ y_2\\
\end{pmatrix}, \begin{pmatrix}
y_1\\ y_2\\
\end{pmatrix}\right\rangle_{\mathbb H} = \frac{1}{\alpha}\Big(\|A^{1/2}y_1\|^2_H + \frac{\alpha^2}{2}\|y_1\|^2_H +\alpha\langle y_1, y_2\rangle_H + \|y_2\|^2_H\Big).
\end{split}
\end{equation}
We first find the maximal $\gamma\ge 0$ such that for all $\begin{pmatrix}y_1\\ y_2\\
\end{pmatrix}\in {\mathbb H}$:
\begin{equation}
\label{03/10/2012(2)}
\frac{1}{\alpha}\Big(\|A^{1/2}y_1\|^2_H + \frac{\alpha^2}{2}\|y_1\|^2_H +\alpha\langle y_1, y_2\rangle_H +\|y_2\|^2_H\Big) \ge \gamma (\|A^{1/2}y_1\|^2_H +\|y_2\|^2_H).
\end{equation}
If $\gamma=1/\alpha$,  inequality (\ref{03/10/2012(2)}) becomes
\[
\frac{\alpha^2}{2}\|y_1\|^2_H +\alpha\langle y_1, y_2\rangle_H \ge 0,\]
which clearly does not hold for all  $\begin{pmatrix}y_1\\ y_2\\
\end{pmatrix}\in {\mathbb H}$, e.g., let $y_2=-\alpha y_1$, $y_1\not= 0$. Therefore, $\gamma\in [0, 1/\alpha)$ and  inequality (\ref{03/10/2012(2)}) becomes
\[
(1-\alpha\gamma)\|A^{1/2}y_1\|^2_H +\frac{\alpha^2}{2}\|y_1\|^2_H +\alpha \langle y_1, y_2\rangle_H +(1-\alpha\gamma)\|y_2\|^2_H\ge 0.\]
On the other hand, for fixed $y_1$, we have
\[
\begin{split}
\min_{y_2\in H} \Big\{\alpha\langle y_1, y_2\rangle_H + (1-\alpha\gamma)\|y_2\|^2_H\Big\}&= \min_{y_2\in H} \Big\{ (1-\alpha\gamma)\Big\|y_2+ \frac{\alpha y_1}{2(1-\alpha\gamma)}\Big\|^2_H - \frac{\alpha^2\|y_1\|^2_H}{4(1-\alpha\gamma)} \Big\}\\
&= -\frac{\alpha^2}{4(1-\alpha\gamma)}\|y_1\|^2_H.
\end{split}
\]
Therefore, the required $\gamma\in [0, 1/\alpha)$ should be such that for $y_1\in {\mathscr D}(A^{1/2})$,
\begin{equation}
\label{07/09/15(1)}
(1-\alpha\gamma)\|A^{1/2}y_1\|^2_H \ge \frac{\alpha^2}{2}\Big(\frac{1}{2(1-\alpha\gamma)}-1\Big)\|y_1\|^2_H.
\end{equation}
Since $A$ is self-adjoint, it is well known that
\[
|\omega_s(-A)| = \inf_{y_1\not= 0}\frac{\|A^{1/2}y_1\|^2_H}{\|y_1\|^2_H}.\]
This implies, in addition to (\ref{07/09/15(1)}),  that one is equivalently looking for the maximum value $\gamma\in [0, 1/\alpha)$ such that
\[
|\omega_s(-A)| \ge \frac{\alpha^2}{4}\Big(\frac{1}{(1-\alpha\gamma)^2}-\frac{2}{1-\alpha\gamma}\Big),\]
or,
\[
\theta \ge \frac{1}{(1-\alpha\gamma)^2}-\frac{2}{1-\alpha\gamma}.\]
This easily gives that
\[
\gamma_{-}= \frac{1}{\alpha}\cdot\frac{\sqrt{1+\theta}}{1+\sqrt{1+\theta}}.\]
In a similar way, the expression for $\gamma_+$ can be obtained by looking for a minimum value $\gamma>0$ such that for all $\begin{pmatrix}y_1\\ y_2\\ \end{pmatrix}\in {\mathbb H}$:
\[
\frac{1}{\alpha}\Big(\|A^{1/2}y_1\|^2_H + \frac{\alpha^2}{2}\|y_1\|_H^2 +\alpha\langle y_1, y_2\rangle_H +\|y_1\|^2_H\Big) \le \gamma \Big(\|A^{1/2}y_1\|^2_H +\|y_1\|^2_H\Big).\]

To prove the final part, we consider the mild solution $y$ of the problem (\ref{12/09/17(02)}). Then from  Lyapunov equation (\ref{12/06/2012(51)}), we have
\[
\frac{d}{dt}\langle {P}y(t), y(t)\rangle_{\mathbb H}=-\|y(t)\|^2_{\mathbb H},\hskip 20pt t\ge 0,\]
which, in addition to (\ref{12/06/2012(5356)}), immediately implies that
\[
\frac{d}{dt}\langle{P}y(t), y(t)\rangle_{\mathbb H}= -\|y(t)\|^2_{\mathbb H}\le -\frac{1}{\gamma_+}\langle {P}y(t), y(t)\rangle_{\mathbb H},\hskip 20pt t\ge 0.\]
A simple calculation further yields that
\[
\langle{P}y(t), y(t)\rangle_{\mathbb H}\le e^{-\frac{1}{\gamma_+}t}\langle{P}y(0), y(0)\rangle_{\mathbb H}\le \gamma_+ e^{-\frac{1}{\gamma_+}t}\|y(0)\|^2_{\mathbb H},\hskip 20pt t\ge 0.\]
Hence, we have
\[
\|y(t)\|^2_{\mathbb H}\le \frac{1}{\gamma_-} \langle{P}y(t), y(t)\rangle_{\mathbb H}\le \frac{\gamma_+}{\gamma_-}e^{-\frac{1}{\gamma_+}t}\|y(0)\|^2_{\mathbb H},\hskip 20pt t\ge 0,\]
as desired.

Next, we consider the case that $B$ is a linear unbounded operator. We first establish a useful lemma.

\begin{lemma}
\label{27/09/17(10)}
Assume that there exist constants $\gamma\ge 0$, $\alpha>0$ such that
\begin{equation}
\label{21/09/17(1)}
\gamma Re\langle Bv, v\rangle_{H}\le -|Im\langle Bv, v\rangle_{H}|,\hskip 20pt \forall\, v\in {\mathscr D}(B),
\end{equation}
and
\begin{equation}
\label{21/09/17(2)}
 Re\langle Bv, v\rangle_{H}\le -2\alpha \|v\|^2_{H},\hskip 20pt \forall\, v\in {\mathscr D}(B).
\end{equation}
 Let $0<\delta<\alpha$ and $ \delta-\alpha<a\le 0$. If the inequality
\[
\inf_{y\in {\mathscr D}(\Lambda),\,\|y\|_{\mathbb H}=1}\|(a + ib)y-\Lambda y\|_{\mathbb H}<\delta,\hskip 20pt b\in {\mathbb R}\]
holds, then
\[
|b|<\alpha\cdot\frac{3\delta+(\delta-a)\gamma}{\alpha -\delta+a}.
\]
\end{lemma}

\begin{proof}
Let $y =\displaystyle{u\choose v}\in {\mathscr D}(\Lambda)$ with $\|y\|_{\mathbb H}^2 = \|A^{1/2}u\|^2_H + \|v\|^2_H=1$ and for $a\in (\delta-\alpha, 0]$, $b\in {\mathbb R}$,
\begin{equation}
\label{21/09/17(5)}
\|(a+ib)y -\Lambda y\|_{\mathbb H} = \Big\|{au+ibu -v\choose Au + av +ibv -Bv}\Big\|_{\mathbb H}<\delta.
\end{equation}
From (\ref{21/09/17(5)}), it easily follows that
\[
|\langle (a+ib)y-\Lambda y, y\rangle_{\mathbb H}|\le \|((a+ib)I-\Lambda)y\|_{\mathbb H}\cdot \|y\|_{\mathbb H}<\delta,
\]
which yields
\begin{equation}
\label{21/09/17(61)}
|\langle (a+ib)y, y\rangle_{\mathbb H} + 2iIm\langle u, Av\rangle_{H} -\langle Bv, v\rangle_{H}|<\delta,
\end{equation}
and further, by considering the real parts in (\ref{21/09/17(61)}),
\begin{equation}
\label{21/09/17(30)}
|Re\langle Bv, v\rangle_{H}-a|<\delta.
\end{equation}
Since $\|A^{1/2}u\|^2_{H} + \|v\|^2_{H}=1$ and $a>\delta -\alpha$, we have by virtue of (\ref{21/09/17(2)}) and (\ref{21/09/17(30)}) that
\begin{equation}
\label{21/09/17(64)}
|1 - 2\|A^{1/2}u\|^2_{H}| =|1-2 +2\|v\|^2_{H}| \ge \frac{Re\langle Bv, v\rangle_{H}}{\alpha} +1 > 1- \frac{\delta-a}{\alpha}.
\end{equation}
On the other hand, we obtain from (\ref{21/09/17(5)}) the estimate
\begin{equation}
\label{21/09/17(69)}
|\langle aA^{1/2}u+ibA^{1/2}u -A^{1/2}v, A^{1/2}u\rangle_{H}|\le \|aA^{1/2}u+ibA^{1/2}u -A^{1/2}v\|_{H}\cdot \|A^{1/2}u\|_{H}<\delta,
\end{equation}
hence, by taking imaginary parts,
\begin{equation}
\label{21/09/17(63)}
|Im\langle u, Av\rangle_{H} + b\|A^{1/2}u\|^2_{H}|<\delta.
\end{equation}
Last, by considering the imaginary part of (\ref{21/09/17(61)}), we have
\[
|b + 2Im\langle u, Av\rangle_{H} - Im\langle Bv, v\rangle_{H}|<\delta,\]
which, in addition to (\ref{21/09/17(63)}), yields
\[
\begin{split}
|b|\cdot |1-2\|A^{1/2}u\|^2_{H}|&-|Im\langle Bv, v\rangle_{H}|\\
 &\le |b(1-2\|A^{1/2}u\|^2_{H})-Im\langle Bv, v\rangle_{H}|\\
&\le |2Im\langle u, Av\rangle_{H} + 2b\|A^{1/2}u\|^2_{H}| + |b + 2Im\langle u, Av\rangle_{H} - Im\langle Bv, v\rangle_{H}|\\
&\le 3\delta.
\end{split}
\]
This relation implies, together with (\ref{21/09/17(1)}),  (\ref{21/09/17(30)}) and (\ref{21/09/17(64)}),  that
\[
|b|\Big(1-  \frac{\delta-a}{\alpha}\Big)\le 3\delta -\gamma Re\langle Bv, v\rangle_{H} < 3\delta + (\delta-a)\gamma,\]
and the assertion follows.
\end{proof}

To proceed further, we need the powerful Gearhart-Pr\"uss-Greiner's lemma whose proof is referred to Theorem 1.11 and Exercise 1.13, pp. 302-304, in \cite{kjern00}. For any linear operator $A$, let $R(\lambda, A)$ denote the resolvent operator of $A$, $\lambda\in \rho(A)$.

\begin{lemma} {\bf (Gearhart-Pr\"uss-Greiner)}
  For a strongly continuous semigroup $e^{tA}$, $t\ge 0$, with generator $A$ on a Hilbert space $H$, its growth bound $\omega_g(A)$ is given by
 \[
 \omega_g(A) = \inf\Big\{a>\omega_s(A):\, \sup_{b\in {\mathbb R}}\|R(a+ib, A)\|<\infty\Big\},\]
 where $\omega_s(A)$ is the spectral bound of $A$.
 In particular,  $e^{tA}$, $t\ge 0$, is exponentially stable if and only if the half-plane $\{\lambda=a+ib\in {\mathbb C}: a >0,\, b\in {\mathbb R}\}$ is contained in the resolvent set $\rho(A)$ of $A$ with the resolvent satisfying
\[
\sup_{a>0,\,b\in {\mathbb R}}\|R(a+ib, A)\|<\infty.\]
\end{lemma}

\begin{proposition}
\label{28/09/17(1)}
Assume that (\ref{21/09/17(1)}) and (\ref{21/09/17(2)}) hold. Let $\Lambda$ be the operator in (\ref{12/09/17(02)}) which generates a contraction $C_0$-semigroup $e^{t{\Lambda}}$, $t\ge 0$, on ${\mathbb H}$.
Then the growth bound $w_g(\Lambda)$ of $\Lambda$ satisfies
\[
w_g(\Lambda) \le \max\{w_s(\Lambda), -\alpha\}<0.\]
\end{proposition}
\begin{proof}
First note that since $\Lambda$ generates a contraction semigroup, it is immediate that
\[
\{\lambda\in {\mathbb C}:\, Re\,\lambda>0\}\subset \rho(\Lambda),\]
i.e., $w_s(\Lambda)\le 0$.

To show $w_s(\Lambda)<0$, we first verify by contraction that $\sigma(\Lambda)\cap i{\mathbb R}=\emptyset.$ Assume that there exists $b_0\in {\mathbb R}$ such that $ib_0\in \sigma(\Lambda)$. If $b_0=0$, then $0\in \sigma(\Lambda)$, a fact which contradicts the assumption $\Lambda^{-1}\in {\mathscr L}({\mathbb H})$. On the other hand, if $b_0\not= 0$, it is always possible to find a number $\delta>0$ small enough such that $|b_0|\ge \alpha \displaystyle\frac{3\delta +\delta \gamma}{\alpha-\delta}$ in Lemma \ref{27/09/17(10)}, and so
\[
\inf_{{y\in {\mathscr D}(\Lambda)},\,\|y\|_{\mathbb H}=1}\|ib_0 y  - \Lambda y\|_{\mathbb H}\ge \delta,\]
which, by definition, implies that $ib_0\notin \sigma_{ap}(\Lambda)$. Here $\sigma_{ap}(\Lambda)$ is the approximate point spectrum set of $\Lambda$, defined by
\[
\sigma_{ap}(\Lambda) = \{\lambda\in {\mathbb C}: \lambda I-\Lambda\,\,\hbox{is not injective or}\,\, {\mathscr R}(\lambda I - \Lambda)\,\,\hbox{is not closed in}\,\, {\mathbb H}\},
\]
where ${\mathscr R}(\lambda I - \Lambda)$ is the range of operator $\lambda I -\Lambda$.
But this is a contradiction, since $ib_0\in \partial \sigma(\Lambda)$ which is a subset of $\sigma_{ap}(\Lambda)$ (cf. Proposition 1.10, pp. 242-243, in \cite{kjern00}). Hence, $i{\mathbb R}\in \rho(\Lambda)$.
Now, let $a=0$ and $|b|\ge \alpha\cdot \displaystyle\frac{3\delta +\delta\gamma}{\alpha-\delta}$ for $0<\delta<\alpha$ in Lemma \ref{27/09/17(10)}, it follows immediately that
\[
\inf_{{y\in {\mathscr D}(\Lambda)},\,\|y\|_{\mathbb H}=1}\|iby  - \Lambda y\|_{\mathbb H}\ge \delta,\hskip 15pt \hbox{or equivalently,}\hskip 15pt \|R(ib, \Lambda)\|\le \frac{1}{\delta}.\]
 On the other hand, $\|R(\lambda, \Lambda)\|$, $\lambda\in \rho(\Lambda)$, is bounded on any compact set of $i{\mathbb R}$. Hence, $w_s(\Lambda)<0$.

Last, according to Gearhart-Pr\"uss-Greiner's lemma, we want to show
\begin{equation}
\label{27/09/17(60)}
a+i{\mathbb R}\subset \rho(\Lambda)\hskip 15pt \hbox{and}\hskip 15pt \sup_{b\in {\mathbb R}}\|R(a+ib, \Lambda)\|<\infty\hskip 15pt \hbox{for all}\hskip 10pt a>\max\{w_s(\Lambda), -\alpha\}.
\end{equation}
Indeed, if $a>\max\{w_s(\Lambda), -\alpha\}$ and
\[
|b|\ge \alpha\cdot \frac{3\delta +(\delta-a)\gamma}{\alpha-\delta+a}\hskip 15pt \hbox{for}\hskip 10pt 0<\delta<\alpha\]
 in Lemma \ref{27/09/17(10)}, then $a+i{\mathbb R}\subset \rho(\Lambda)$ and
\[
\sup_{\{b:\, |b|\ge \alpha \frac{3\delta +(\delta-a)\gamma}{\alpha-\delta+a}\}}\|R(a+ib, \Lambda)\|\le \frac{1}{\delta}.\]
On the other hand, $\|R(\lambda, \Lambda)\|$, $\lambda\in \rho(\Lambda)$, is bounded on any compact set of ${\mathbb C}$. Hence, the desired (\ref{27/09/17(60)}) holds.  The proof is complete now.
\end{proof}

\begin{corollary}
\label{13/10/17(1)}
Under the same conditions as in Proposition \ref{28/09/17(1)},
\begin{enumerate}
\item[(i)] if $\gamma\not= 0$, then
\[
w_g(\Lambda)\le \nu,\]
where $\nu\in (-\alpha, 0)$ is the unique solution of the equation
\[
\nu^2 + \Big(\frac{\nu\gamma\alpha}{\alpha+\nu}\Big)^2 =\|\Lambda^{-1}\|^{-2};\]
\item[(ii)] if $\gamma=0$, then
\[
w_g(\Lambda) \le \max\{-\alpha, -\|\Lambda^{-1}\|^{-1}\}.\]
\end{enumerate}
\end{corollary}
\begin{proof} For $\gamma\not= 0$, if $a\in (-\alpha, 0]$ and $|b|> \displaystyle\frac{-\alpha\gamma a}{a+\alpha}$, then there exists a sufficiently small $\delta>0$ such that $-\alpha +\delta<a\le 0$ and
\[
|b|\ge \alpha\cdot \frac{3\delta +(\delta-a)\gamma}{\alpha-\delta+a}.\]
By virtue of Lemma \ref{27/09/17(10)}, it thus implies that  $a+ib\in \rho(\Lambda)$. That is, the curve
\[
\Big(\nu, \frac{\nu\gamma\alpha}{\alpha+\nu}\Big), \hskip 15pt \nu\in (-\alpha, 0),\]
 is contained in the resolvent set.
On the other hand, it is easy to know that the disk with radius $\|\Lambda^{-1}\|^{-1}$ centered at zero is contained in the resolvent set (cf. Theorem 2.3, p. 274 in \cite{atdl80}). We thus obtain the desired (i) by intersecting the two curves. For (ii), we choose first $\gamma>0$, use (i) and then take the limit as $\gamma\to 0$.
\end{proof}

\begin{corollary}
\label{05/11/17(20)}
Under the same conditions as in Proposition \ref{28/09/17(1)},
\begin{enumerate}
\item[(i)] if $\gamma\not= 0$, then
\[
w_g(\Lambda)\le \nu,\]
where $\nu\in (-\alpha, 0)$ is the unique solution of the equation
\[
\nu^2 + \Big(\frac{\nu\gamma\alpha}{\alpha+\nu}\Big)^2 =(\|A^{-1/2}BA^{-1/2}\| + 2\|A^{-1/2}\|)^{-2};\]
\item[(ii)] if $\gamma=0$, then
\[
w_g(\Lambda) \le \max\{-\alpha, -(\|A^{-1/2}BA^{-1/2}\| + 2\|A^{-1/2}\|)^{-1}\}.\]
\end{enumerate}
\end{corollary}
\begin{proof}
For any $\displaystyle{u\choose v}\in {\mathscr D}(\Lambda_0)$, we have
\[
\begin{split}
\Big\|\Lambda^{-1}_0{u\choose v}\Big\|_{\tilde{\mathbb H}} &= \Big(\|A^{-1/2}BA^{-1/2}u -A^{-1/2}v\|^2_H + \|A^{-1/2}u\|^2_H\Big)^{1/2}\\
&\le \|A^{-1/2}BA^{-1/2}u - A^{-1/2}v\|_H + \|A^{-1/2}u\|_H\\
&\le \|A^{-1/2}BA^{-1/2}\|\|u\|_H + \|A^{-1/2}\|\|v\|_H + \|A^{-1/2}\|\|u\|_H\\
&\le \Big(\big(\|A^{-1/2}BA^{-1/2}\| + \|A^{-1/2}\|\big)^2 + \|A^{-1/2}\|^2\Big)^{1/2}\cdot (\|u\|^2_H + \|v\|^2_H)^{1/2}\\
&\le (\|A^{-1/2}BA^{-1/2}\|+ 2\|A^{-1/2}\|)(\|u\|^2_H + \|v\|^2_H)^{1/2}.
\end{split}
\]
Since ${\mathscr D}(\Lambda_0)$ is dense in $\tilde{\mathbb H}$, we thus have
\[
\|\Lambda^{-1}_0\|_{{\mathscr L}(\tilde{\mathbb H})}\le \|A^{-1/2}BA^{-1/2}\| +2\|A^{-1/2}\|,\]
i.e.,
\[
\|\Lambda^{-1}\|_{{\mathscr L}({\mathbb H})}\le \|A^{-1/2}BA^{-1/2}\| +2\|A^{-1/2}\|.\]
By virtue of Corollary \ref{13/10/17(1)}, we obtain the desired result.
\end{proof}

Next we employ Lemma \ref{27/09/17(10)} to estimate the resolvent $R(ib, \Lambda)$, $b\in {\mathbb R}$, which will play an important role in dealing with stochastic second-order Cauchy problems with delay.

\begin{lemma}
\label{16/10/17(1)}
Under the same conditions as in Theorem \ref{05/10/2017(1)} and Lemma  \ref{27/09/17(10)}, we have for every $0<c<1$,
\[
\|R(ib, \Lambda)\| \le
\begin{cases}
\displaystyle\frac{\|\Lambda^{-1}\|}{1-c}\hskip 20pt &\hbox{for}\hskip 15pt |b|\le \displaystyle\frac{c}{\|\Lambda^{-1}\|},\\
\displaystyle\frac{(3+\gamma)\alpha\|\Lambda^{-1}\| +c}{\alpha c}\hskip 20pt &\hbox{for}\hskip 15pt |b|> \displaystyle\frac{c}{\|\Lambda^{-1}\|},
\end{cases}
\]
\end{lemma}
\begin{proof}
Since $\Lambda$ is invertible, the resolvent of $\Lambda$ is given by
\[
R(\lambda, \Lambda) = \sum^\infty_{n=1} \lambda^n \Lambda^{-(n+1)}\hskip 20pt \hbox{for}\hskip 15pt |\lambda|< \frac{1}{\|\Lambda^{-1}\|}.\]
Moreover, for $0<c<1$, we have that
\[
\|R(\lambda, \Lambda)\| \le \|\Lambda^{-1}\|\sum^\infty_{n=0} |\lambda|^n\|\Lambda^{-1}\|^n =\frac{\|\Lambda^{-1}\|}{1-c}\hskip 15pt \hbox{for any}\hskip 15pt |\lambda|\le  \frac{c}{\|\Lambda^{-1}\|}.
\]
For $|\lambda|> c/\|\Lambda^{-1}\|$, let $0<\delta<\alpha$ and we have by Lemma \ref{27/09/17(10)} that
\[
\|R(ib, \Lambda)\|<\frac{1}{\delta}\hskip 15pt \hbox{for}\hskip 15pt |b|> \displaystyle\frac{\alpha(3\delta+\gamma)}{\alpha-\delta}.\]
Now we have to look for one $\delta\in (0, \alpha)$ such that
\[
\frac{\alpha(3\delta+\gamma)}{\alpha-\delta} =  \frac{c}{\|\Lambda^{-1}\|}.\]
But this is possible when
\[
\delta = \frac{\alpha c}{\alpha(3+\gamma)\|\Lambda^{-1}\|+c}\]
which concludes the proof.
\end{proof}

\begin{corollary}
\label{06/11/17(2)}
Under the same conditions as in Lemma \ref{16/10/17(1)}, we have
\[
\|R(ib, \Lambda)\|\le \frac{2\alpha(3+\gamma)\kappa^{-1} +1}{\alpha}\]
for all $b\in {\mathbb R}$ and every lower bound $\kappa>0$ of $\Lambda$.
\end{corollary}
\begin{proof}
Since $\|\Lambda^{-1}\|\le \kappa^{-1}$, we choose $c=1/2$ in Lemma \ref{16/10/17(1)} and obtain
\[
\|R(ib, \Lambda)\|\le \frac{2\alpha(3+\gamma)\|\Lambda^{-1}\| + 1}{\alpha}\le \frac{2\alpha(3+\gamma)\kappa^{-1} + 1}{\alpha}\]
as desired.
\end{proof}

\section{\large Asymptotic Behavior of Solutions with Delay}

Now let $r>0$ and consider a linear Cauchy problem with memory in the Hilbert space $H$,
\begin{equation}
\label{12/09/17(134896789)}
\begin{cases}
d\Big(\displaystyle\frac{du(t)}{dt}\Big)  + Au(t)dt = Bu'(t)dt+ Mu_tdt + Nu'_tdt,\hskip 15pt t\ge 0,\\
u(0)= \phi_{0, 1}\in V,\,\,\,u'(0)=\phi_{0, 2}\in H,\\
u_0= \phi_{1,1}\in L^2([-r, 0], V),\,\,\, u'_0=\phi_{1, 2}\in L^2([-r, 0], H),
\end{cases}
\end{equation}
where $A$, $B$ and $V$ are given as in Section 2 and the mappings $M: W^{1, 2}([-r, 0], V)\to H$, $N: W^{1, 2}([-r, 0], H)\to H$ are two bounded linear operators. Here $W^{1, 2}([-r, 0], V)$ and $W^{1, 2}([-r, 0], H)$ are the standard Sobolev spaces of functions from $[-r, 0]$ into $V$ and $H$, respectively. We strengthen the condition on $M$, $N$ by assuming further that $M$, $N$ are given by the Riemann-Stieljes integrals of functions of bounded variation $\eta:\, [-r, 0]\to {\mathscr L}(V, H)$ and $\eta:\, [-r, 0]\to {\mathscr L}(H)$, i.e.,
\[
M(\varphi) = \int^0_{-r} d\eta(\theta)\varphi(\theta)\hskip 15pt \forall\,\varphi\in W^{1, 2}([-r, 0], V),\]
and
\[
N(\varphi) = \int^0_{-r} d\zeta(\theta)\varphi(\theta)\hskip 15pt \forall\,\varphi\in W^{1, 2}([-r, 0], H).\]
For each $\lambda\in {\mathbb C}$, we define in connection with $M$ a linear operator $M(e^{\lambda\cdot}): V\to V$ by
\[
M(e^{\lambda\cdot})x = M(e^{\lambda\cdot}x),\hskip 20pt x\in V.\]
Then it is easy to see that $M(e^{\lambda\cdot})\in {\mathscr L}(V)$ and
$$\|M(e^{\lambda\cdot})\|_{{\mathscr L}(V)}\le e^{|\lambda|r}Var(\eta)^0_{-r}$$
where $Var(\eta)^0_{-r}$ is the total variation of $\eta$ on $[-r, 0]$. In a similar way, one can define and show $N(e^{\lambda\cdot})\in {\mathscr L}(H)$ and
\[
\|N(e^{\lambda\cdot})\|_{{\mathscr L}(H)}\le e^{|\lambda|r}Var(\zeta)^0_{-r}.\]
By defining $y(t) = \displaystyle{u(t)\choose u'(t)}$, $y_t = \displaystyle{u_t\choose u'_t}$, let us rewrite the problem (\ref{12/09/17(134896789)}) as a first-order  delay differential equation in ${\mathbb H}$,
\begin{equation}
\label{24/05/17(1678)}
\begin{cases}
dy(t) = \Lambda y(t)dt + Fy_tdt,\,\,\,\,t\ge 0,\\
y(0)= \phi_0 = \displaystyle{\phi_{0, 1}\choose \phi_{0, 2}}\in {\mathbb H},\,\,\,y_0 = \phi_1 =\displaystyle{\phi_{1, 1}\choose \phi_{1, 2}}\in L^2([-r, 0], {\mathbb H}),
\end{cases}
\end{equation}
where $\Lambda$ is given as in (\ref{12/09/17(02)}) and delay operator matrix $F = \displaystyle\left(\begin{array}{cc}
0& 0\\
M& N
\end{array}\right)$ is a bounded linear operator from $W^{1, 2}([-r, 0], V)\times W^{1, 2}([-r, 0, H)$ into ${\mathbb H}$ given by
\[
F\left(\begin{array}{c}
\varphi_1\\
\varphi_2
\end{array}\right) = \left(\begin{array}{cc}
0& 0\\
M& N
\end{array}\right)\left(\begin{array}{c}
\varphi_1\\
\varphi_2
\end{array}\right) = \left(\begin{array}{c}
0\\
\displaystyle\int^0_{-r} d\eta(\theta)\varphi_1(\theta) + \int^0_{-r} d\zeta(\theta)\varphi_2(\theta)
\end{array}\right)
\]
for any $\varphi_1\in W^{1, 2}([-r, 0], V)$, $\varphi_2\in W^{1, 2}([-r, 0], H)$.

Let ${\cal H}_2$, or simply ${\cal H}$, denote the Hilbert space ${\mathbb H}\times L^2([-r, 0], {\mathbb H})$, equipped with the usual product space inner product
\[
\langle \varphi, \psi\rangle_{{\cal H}_2} := \langle \varphi_{0}, \psi_{0}\rangle_{\mathbb H} + \int^0_{-r}\langle \varphi_{1}(\theta), \psi_{1}(\theta)\rangle_{\mathbb H}d\theta,\]
for any $\varphi=(\varphi_{0}, \varphi_{1})$, $\psi=(\psi_{0}, \psi_{1})\in {\cal H}_2$.
It may be shown  that for any $\phi=(\phi_0, \phi_1)\in {\cal H}_2$, the equation  (\ref{24/05/17(1678)}) has a unique mild solution $y(t, \phi)$.  For any $x\in {\mathbb H}$, we introduce the  {\it fundamental solution\/} or {\it  Green operator\/} $G(t): (-\infty, \infty)\to {\mathscr L}({\mathbb H})$ of  ({\ref{24/05/17(1678)}) by
\begin{equation}
\label{25/05/06(1)}
G(t)x = \begin{cases}
y(t, \phi),\hskip 15pt &t\ge 0,\\
0,\hskip 15pt &t< 0,
\end{cases}
\end{equation}
where $\phi= (x, 0)$, $x\in {\mathbb H}$.  It turns out
(cf. \cite{kliu08(0)})
 that $G(t)$, $t\ge 0$, is a strongly continuous one-parameter family of bounded linear operators on ${\mathbb H}$.
For each function $\varphi:\, [-r, 0]\to {\mathbb H}$, we define its right extension function  $\vec{\varphi}$  by
\begin{equation}
\label{12/01/08(1)}
\vec{\varphi}: \,\, [-r, \infty)\to {\mathbb H},\,\,\,\, \vec{\varphi}(t) =
\begin{cases}
\varphi(t),&\hskip 15pt -r\le t\le 0,\\
0,&\hskip 15pt 0<t<\infty.
\end{cases}
\end{equation}
By virtue of (\ref{12/01/08(1)}), it may be shown (cf. \cite{kliu08(0)}) that the mild solution $y(t, \phi)$ of (\ref{24/05/17(1678)}) is represented explicitly by the variation of constants formula
\begin{equation}
\label{25/05/06(4)}
y(t) = G(t)\phi_0 + \int^t_{0} G(t-s)F(\vec{\phi_1})_sds,\hskip 15pt t\ge 0,
\end{equation}
and $y(t) =\phi_1(t)$, $t\in [-r, 0)$. It is useful to introduce the so-called structure operator $S$ defined on the space $L^2([-r, 0]; {\mathbb H})$ by
\begin{equation}
\label{28/06/06(101)}
\begin{split}
(S\varphi)(\theta) =F\vec{\varphi}_{-\theta},\hskip 10pt \theta\in [-r, 0],\hskip 15pt \forall\, \varphi(\cdot)\in W^{1,2}([-r, 0]; {\mathbb H}).
\end{split}
\end{equation}
It is not difficult to show that $S$ can be extended to a  linear and bounded operator from $L^2_r := L^2([-r, 0]; {\mathbb H})$ into itself.  Moreover,  the variation of constants formula for the mild solution of (\ref{24/05/17(1678)}) may be rewritten as
\begin{equation}
\label{29/01/08(1)}
\begin{cases}
y(t) = G(t)\phi_0 + \displaystyle\int^{0}_{-r} G(t+\theta)(S\phi_1)(\theta)d\theta,\hskip 15pt t\ge 0,\\
y(0)=\phi_0,\,\, y(\theta) =\phi_1(\theta),\,\,\,\,\,\theta\in [-r, 0].
\end{cases}
\end{equation}

The mild solution $y(t, \phi)$ of (\ref{24/05/17(1678)}) allows us to introduce a $C_0$-semigroup on ${\cal H}_2$. Indeed, we define a mapping ${\cal S}(t)$, $t\ge 0$, associated with $y(t, \phi)$,  by
\begin{equation}
\label{09/08/07(2)}
{\cal S}(t)\phi =(y(t, \phi), y_t(\cdot, \phi)),\,\,\,\,\,\,t\ge 0,\,\,\,\,\phi\in {\cal H}.
\end{equation}
 It can be shown (cf. \cite{sn1988})
that the mapping ${\cal S}(t)$, $t\ge 0$, is a $C_0$-semigroup with some infinitesimal generator ${\cal A}$ on the space ${\cal H}_2$. Moreover, the operator ${\cal A}$ may be explicitly specified as follows.
\begin{proposition}
\label{09/08/07(10)}
 The  generator ${\cal A}$ of the $C_0$-semigroup ${\cal S}(t)$, or denote it by $e^{t{\cal A}}$, is described by
\begin{equation}
\label{09/08/07(4)}
{\mathscr D}({\cal A})=\Big\{\phi=(\phi_0, \phi_1)\in {\cal H}_2:\, \phi_1\in W^{1, 2}([-r, 0]; {\mathbb H}),\,\phi_1(0)=\phi_0\in {\mathscr D}(\Lambda)\Big\},
\end{equation}
\begin{equation}
\label{09/08/07(5)}
{\cal A}\phi=\Big(\Lambda\phi_0 + F\phi_1, \frac{d\phi}{d\theta}(\theta)\Big)\,\,\,\,\,\hbox{for any}\,\,\,\,\phi=(\phi_0, \phi_1)\in {\mathscr D}({\cal A}).
\end{equation}
\end{proposition}

The proof of the following proposition is referred to Proposition 2.2 in \cite{kliu11}.

\begin{proposition}
\label{25/09/08(10)}
For  the Green operator $G(t)$, $t\in {\mathbb R}$, and $C_0$-semigroup $e^{t{\cal A}}$, $t\ge 0$, in (\ref{09/08/07(2)}),
 the following relations are equivalent:
\begin{enumerate}
\item[(i)] the semigroup $e^{t{\cal A}}$ is exponentially stable, i.e., there exist constants $M>0$ and $\mu>0$ such that $\|e^{t{\cal A}}\|\le Me^{-\mu t}$, $t\ge 0$.
\item[(ii)] the Green operator $G(t)$ is exponentially stable, i.e., there exist constants $M>0$ and $\mu>0$ such that $\|G(t)\|\le Me^{-\mu t}$, $t\ge 0$.
\end{enumerate}
\end{proposition}

Now we consider the asymptotic behavior of $C_0$-semigroup $e^{t{\cal A}}$, $t\ge 0$, on ${\cal H}$. Recall that for any $\lambda\in {\mathbb C}$, $F(e^{\lambda\cdot})y := F(e^{\lambda\cdot}y)$, $y\in {\mathbb H}$, defines a bounded linear operator $F(e^{\lambda\cdot})$ on ${\mathbb H}$. We first state a proposition whose proof is referred to Theorem 5.5, pp. 104-105, in \cite{absp05}.

\begin{proposition}
Assume that $\Lambda$ generates an exponentially stable $C_0$-semigroup $e^{t{\Lambda}}$, $t\ge 0$, on ${\mathbb H}$, i.e., $\omega_g(\Lambda)<0$, and let $a\in (\omega_g(\Lambda), 0]$ such that
\[
\alpha_{a, n} := \sup_{b\in {\mathbb R}}\|[F(e^{(a+ib)\cdot})R(a+ib, \Lambda)]^n\|<\infty\hskip 20pt \hbox{for any}\hskip 16pt n\in {\mathbb N}.\]
Further, if the series $\sum^\infty_{n=0}\alpha_{a, n}<\infty$, then $\omega_g({\cal A})<a\le 0$. That is, ${\cal A}$ generates in this case an exponentially stable $C_0$-semigroup $e^{t{\cal A}}$, $t\ge 0$, on ${\cal H}$.
\end{proposition}

\begin{corollary}
\label{05/11/17(50)}
Assume that the growth bound of $\Lambda$ satisfies $\omega_g(\Lambda)<0$, and let $a\in (\omega_g(\Lambda), 0]$. If
  \[
  \sup_{b\in {\mathbb R}}\|F(e^{({a+ib})\cdot})\|< \frac{1}{\sup_{b\in {\mathbb R}}\|R(a+ib, \Lambda)\|},\]
  then $\omega_g({\cal A})<a\le 0$.
\end{corollary}
\begin{proof}
Defining $q_a = \sup_{b\in{\mathbb R}}\|F(e^{(a+ib)\cdot})R(a+ib, \Lambda)\|<1$, we obtain that $\alpha_{a, n}\le q^n_a$, hence the series $\sum^\infty_{n=0}\alpha_{a, n}$ is convergent.
\end{proof}

\section{\large Stochastic Systems Driven by Wiener Processes}

Let $\{\Omega, {\mathscr F}, {\mathbb P}\}$ be a
 probability space equipped with some filtration $\{\mathscr{F}_{t}\}_{t\geq 0}$. Let $K$ be a separable Hilbert space and  $\{W_Q(t),\,t\ge 0\}$ denote a $Q$-Wiener process with respect to $\{\mathscr{F}_{t}\}_{t\geq 0}$ in $K$, defined on  $\{\Omega, {\mathscr F}, {\mathbb P}\}$, with covariance operator $Q$, i.e.,
\[
\mathbb{E}\langle W_Q(t), x\rangle_K\langle W_Q(s), y\rangle_K = (t\wedge s)\langle Qx, y\rangle_K\,\,\,\,\hbox{for all}\,\,\,\,\, x,\,\,y\in K,\]
where $Q$ is a positive, self-adjoint and trace class operator on $K$. We frequently call $W_Q(t)$, $t\ge 0$, a $K$-valued $Q$-Wiener process with respect to $\{\mathscr{F}_{t}\}_{t\geq 0}$ if the trace Tr$\,Q<\infty$.
We introduce a subspace $K_Q={\mathscr R}(Q^{1/2})\subset K$, the range of $Q^{1/2}$, which is a Hilbert space endowed with the inner product
\[
\langle u, v\rangle_{K_Q} =\langle Q^{-1/2}u, Q^{-1/2}v\rangle_K\hskip 10pt\hbox{for any}\hskip 10pt  u,\,\,v\in K_Q.\]
Let $X$ be a separable Hilbert space and ${\mathscr L}_2(K_Q, X)$ denote the space of all Hilbert-Schmidt operators from $K_Q$ into $X$. Then ${\mathscr L}_2(K_Q, X)$ turns out to be a separable Hilbert space, equipped with the norm
\[
\|\Psi\|^2_{{\mathscr L}_2(K_Q, X)} =Tr [\Psi Q^{1/2}(\Psi Q^{1/2})^*]\hskip 15pt \hbox{for any}\,\,\,\,\Psi\in {\mathscr L}_2(K_Q, X).\]
For arbitrarily given $T\ge 0$, let
$J(t,\omega)$, $t\in[0,T]$, be an ${\mathscr L}_2(K_Q, X)$-valued process, and we define the following norm for arbitrary $t\in[0,T]$,
\begin{equation}
\label{11/02/09(10)}
|J|_t :=\biggl\{\mathbb{E}\int^t_0 Tr\Big[J(s,\omega)Q^{1/2}(J(s,\omega)Q^{1/2})^*\Big]ds\biggr\}^{\frac{1}{2}}.
\end{equation}
In particular, we denote all ${\mathscr L}_2(K_Q, X)$-valued measurable processes $J$, adapted to the filtration $\{{\mathscr F}_t\}_{t\le T}$, satisfying $|J|_T <
\infty$ by ${\cal U}^2\big([0,T]; \,{\mathscr L}_2(K_Q, X)\big)$.

Let $r>0$ and consider a semi-linear Cauchy problem with memory in the Hilbert space $H$,
\begin{equation}
\label{12/09/17(1348967)}
\begin{cases}
d\Big(\displaystyle\frac{du(t)}{dt}\Big)  + Au(t)dt = Bu'(t)dt+ Mu_tdt + Nu'_tdt  + R(u(t), u'(t), u_t, u'_t)dW(t),\hskip 15pt t\ge 0,\\
u(0)= \phi_{0, 1}\in V,\,\,\,u'(0)=\phi_{0, 2}\in H,\\
u_0= \phi_{1,1}\in L^2([-r, 0], V),\,\,\, u'_0=\phi_{1, 2}\in L^2([-r, 0], H),
\end{cases}
\end{equation}
where $A$, $B$, $M$, $N$ and $V$ are given as in Section 4 and $W$ is a $Q$-Wiener process in $K$. The non-linear mapping
\[
R: V\times H\times L^2([-r, 0], V)\times L^2([-r, 0], H)\to {\mathscr L}_2(K_Q, H)\]
 is assumed to be Borel measurable and there exist constants $\alpha_1,\,\alpha_2>0$ and a finite measure $\kappa(\cdot)$ on $[-r, 0]$
 such that
\begin{equation}
\label{08/10/17(2)}
\begin{split}
\|R(\phi)-R(\psi)&\|_{{\mathscr L}_2(K_Q, H)}^2\le \alpha_1\Big(\|\phi_{0, 1}-\psi_{0, 1}\|^2_V + \|\phi_{0, 2}-\psi_{0, 2}\|^2_H\Big)\\
&\,\, + \alpha_2\Big(\int^0_{-r} \|\phi_{1, 1}(\theta) - \psi_{1,1}(\theta)\|^2_V \kappa(d\theta) + \int^0_{-r} \|\phi_{1, 2}(\theta) - \psi_{1,2}(\theta)\|^2_H \kappa(d\theta)\Big)
\end{split}
\end{equation}
for any $\phi = (\phi_0, \phi_1) = (\phi_{0, 1}, \phi_{0, 2}, \phi_{1,1}, \phi_{1, 2}),\, \psi =(\psi_0, \psi_1) =  (\psi_{0, 1}, \psi_{0, 2}, \psi_{1,1}, \psi_{1, 2})\in V\times H\times L^2([-r, 0], V)\times L^2([-r, 0], H)$.

Our strategy here is to apply the techniques developed in the previous sections. Hence, as the first step, let us rewrite this problem into a first-order stochastic delay system.
That is, by defining $y(t) = \displaystyle{u(t)\choose u'(t)}$,  $y_t = \displaystyle{u_t\choose u'_t}$ and a nonlinear mapping $L$ by
\[
L(\phi_0, \phi_1) = {0\choose R(\phi_{0, 1}, \phi_{0, 2}, \phi_{1,1}, \phi_{1,2})}\]
for any $\phi = (\phi_0, \phi_1) = (\phi_{0, 1}, \phi_{0, 2}, \phi_{1,1}, \phi_{1, 2})\in V\times H\times L^2([-r, 0], V)\times L^2([-r, 0], H)$,
 we can transform (\ref{12/09/17(1348967)}) into a first-order stochastic system
\begin{equation}
\label{24/05/17(1)}
\begin{cases}
dy(t) = \Lambda y(t)dt + Fy_tdt + L(y(t), y_t)dW(t),\,\,\,\,t\ge 0,\\
y(0)= \phi_0 = \displaystyle{\phi_{0, 1}\choose \phi_{0, 2}}\in {\mathbb H},\,\,\,y_0 =\phi_1 = {\phi_{1, 1}\choose  \phi_{1, 2}}\in L^2([-r, 0], {\mathbb H}),
\end{cases}
\end{equation}
where $L: {\cal H}\to {\mathscr L}_2(K_Q, {\mathbb H})$ is clearly Borel measurable and  (\ref{08/10/17(2)}) implies immediately that
 \begin{equation}
\label{24/05/17(2)}
\begin{split}
\|L(\phi)-L(\psi)\|^2_{{\mathscr L}_2(K_Q, {\mathbb H})}\le \alpha_1\|\phi_0-\psi_0\|^2_{\mathbb H} + &\,\alpha_2\int^0_{-r} \|\phi_1(\theta)-\psi_1(\theta)\|^2_{L^2([-r, 0], {\mathbb H})}\kappa(d\theta),\\
&\forall\,\phi=(\phi_0, \phi_1),\,\psi=(\psi_0, \psi_1)\in {\cal H}.
\end{split}
\end{equation}
Let ${C}_b({\cal H})$ denote the set of all bounded and continuous real-valued functions on ${\cal H}$ and ${\mathscr P}({\cal H})$ be the space of all probability measures on $({\cal H}, {\mathscr B}({\cal H}))$ where ${\mathscr B}({\cal H})$ is the Borel $\sigma$-algebra on ${\cal H}$.
For any mild solution $y$ of (\ref{24/05/17(1)}), it is well-known that
  $Y(t) := (y(t), y_t)$, $t\ge 0$, is a Markov process in ${\cal H}$.

\begin{definition}\rm
A {\it stationary distribution} for $Y(t) = (y(t), y_t)$, $t\ge 0$, of equation (\ref{24/05/17(1)}) is defined as a probability measure $\mu\in {\mathscr P}({\cal H})$ satisfying
\[
\mu(f) = \mu({\mathbb P}_tf),\,\,\,\,\,t\ge 0,\]
where
\[
\mu(f) := \int_{\cal H}f(\phi)\mu(d\phi)\hskip 15pt \hbox{and}\hskip 15pt {\mathbb P}_tf(\phi) := {\mathbb E}f(Y(t, \phi)),\hskip 15pt f\in {C}_b({\cal H}).\]
\end{definition}

 For ${\mu}_1$, ${\mu}_2\in\mathscr{P}({\cal H})$, define a metric on $\mathscr{P}({\cal H})$ by
\begin{equation}
\label{07/06/17(1)}
d({\mu}_1,{\mu}_2)=\sup_{f\in {\cal M}}\Big|\int_{\cal H}f(\phi){\mu}_1(d\phi)-\int_{\cal H}f(\psi){\mu}_2(d\psi)\Big|,
\end{equation}
where
\[
{\cal M} :=\{f: {\cal H}\rightarrow \mathbb{R},\, |f(\phi)-f(\psi)|\le \|\phi-\psi\|_{\cal H}\,\,\hbox{for any}\,\, \phi,\,\psi\in {\cal H} \;\; \textrm{and} \;\;|f(\cdot)|\le 1\}.\]
Then it is well known (cf. \cite{rmd2003} or \cite{partha67}) that $\mathscr{P}({\cal H})$  is complete under the  metric $d(\cdot, \cdot)$.

\begin{lemma}
\label{25/10/17(2)}
Suppose that for any bounded subset $U$ of ${\cal H}$,

(i) $\lim_{t\to\infty} \sup_{\phi,\,\psi\in U}{\mathbb E}\|Y(t, \phi)-Y(t, \psi)\|^2_{\cal H}=0;$

(ii) $\sup_{t\ge 0}\sup_{\phi\in U}{\mathbb E}\|Y(t, \phi)\|^2_{\cal H}<\infty.$

\noindent
 Then, for initial data $\phi\in {\cal H}$,  process $Y(t, \phi)$, $t\ge 0$, has a stationary distribution.
\end{lemma}
\begin{proof}
 It suffices to show that for any initial data $\phi\in {\cal H}$, $\{{\mathbb P}(\phi,t,\cdot):t\geq0\}$ is Cauchy in the space $\mathscr{P}({\cal H})$ with the metric $d(\cdot, \cdot)$ in (\ref{07/06/17(1)}).
To this end, we need only show that for any fixed $\phi\in {\cal H}$ and $\varepsilon>0$, there exists a time $T>0$ such that
\begin{equation}
\label{25/10/17(1)}
d({\mathbb P}(\phi,t+s,\cdot),{\mathbb P}(\phi,t,\cdot))=\sup_{f\in {\cal M}}|\mathbb{E}f(Y(t+s, \phi))-\mathbb{E}f(Y(t, \phi))| \leq\varepsilon, \;\;\;\;\;\;\; \forall \, t\geq T, \;s>0.
\end{equation}
Indeed, for any $f\in {\cal M}$ and $t,s>0$, we can obtain that
\begin{equation}
\begin{split}
|\mathbb{E}&f(Y(t+s, \phi))-\mathbb{E}f(Y(t, \phi))|\\ &=|\mathbb{E}[\mathbb{E}f(Y(t+s, \phi))|{\mathscr F}_s)]-\mathbb{E}f(Y(t, \phi))|\\
&=\Big|\int_{H}\mathbb{E}f(Y(t, \psi)){\mathbb P}(\phi,s,d\psi)-\mathbb{E}f(Y(t, \phi))\Big|\\
&\leq \int_{H}|\mathbb{E}f(Y(t, \psi))-\mathbb{E}f(Y(t, \phi))|{\mathbb P}(\phi,s,d\psi)\\
&\leq 2{\mathbb P}(\phi,s, {\cal H}^c_R)+\int_{{\cal H}_R}|\mathbb{E}f(Y(t, \psi))-\mathbb{E}f(Y(t, \phi))|{\mathbb P}(\phi,s,d\psi),
\end{split}
\end{equation}
where ${\cal H}_R=\{\phi\in {\cal H}:\|\phi\|_{\cal H}\leq R\}$ and ${\cal H}^c_R={\cal H}-{\cal H}_R$. By virtue of condition (i), there exists a positive number $R$ sufficiently large such that
\begin{equation}
{\mathbb P}(\phi,s, {\cal H}^c_R)<\frac{\varepsilon}{4}, \;\;\;\;\; \forall s>0.
\end{equation}
On the other hand, by virtue of condition (ii), there exists a time $T_2>0$ such that
\begin{equation}
\sup_{f\in {\cal M}}|\mathbb{E}f(Y(t, \psi))-\mathbb{E}f(Y(t, \phi))|\leq \frac{\varepsilon}{2}, \;\;\;\; t\geq T_2.
\end{equation}
Hence, substituting (5.21), (5.22) into (5.20) immediately yields that
$$|\mathbb{E}f(Y(t+s, \phi))-\mathbb{E}f(Y(t, \phi))|\leq\varepsilon, \;\;\;\;\; t\geq T_2, \; s>0.$$
Since function  $f\in {\cal M}$ is arbitrary, we can get the desired result (\ref{25/10/17(1)}). Hence, the transition probability ${\mathbb P}(\phi, t, \cdot)$ of $Y(t, \phi)$ converges weakly to some $\mu\in {\mathscr P}({\cal H})$.
On the other hand, for any $f\in C_b({\cal H})$, one has by the Markovian property of $Y(t, \phi)$, $t\ge 0$ that
\[
{\mathbb P}_{t+s}f(\phi) = {\mathbb P}_t{\mathbb P}_sf(\phi),\hskip 15pt t,\,\,s\ge 0,\,\,\,
\phi\in {\cal H}.\]
Hence, for any fixed $t\ge 0$, as $s\to \infty$, it follows that
\[
\mu(f) = \mu({\mathbb P}_tf),\hskip 20pt f\in C_b({\cal H}).\]
That is, $\mu$ is a stationary distribution for $Y(t) = (y(t), y_t)$, $t\ge 0$, of equation (\ref{24/05/17(1)}) provided that (i) and (ii) above hold.
 The proof is thus complete.
\end{proof}

Recall that $e^{t{\cal A}}$, $t\ge 0$, is the $C_0$-semigroup on ${\cal H}$ given in Proposition \ref{09/08/07(10)}.

\begin{theorem}
\label{05/06/17(80)}
Assume that $\|e^{t{\cal A}}\|\le Me^{-\gamma t}$, $M\ge 1$, $\gamma>0$ for all $t\ge 0$. Suppose that (\ref{24/05/17(2)}) holds  and
\begin{equation}
\label{26/05/17(10)}
2\gamma > 3M^2(\alpha_1 + \alpha_2 e^{2\gamma r}\kappa([-r, 0])),
\end{equation}
then
there is a unique stationary distribution $\mu \in {\mathscr P}({\cal H})$ for $Y(t)=(y(t), y_t)$, $t\ge 0$, of (\ref{24/05/17(1)}).
\end{theorem}
\begin{proof}
By Lemma \ref{25/10/17(2)}, we need only verify the following assertions: for any bounded subset $U$ of ${\cal H}$,

(i) $\lim_{t\to\infty} \sup_{\phi,\,\psi\in U}{\mathbb E}\|Y(t, \phi)-Y(t, \psi)\|^2_{\cal H}=0;$

(ii) $\sup_{t\ge 0}\sup_{\phi\in U}{\mathbb E}\|Y(t, \phi)\|^2_{\cal H}<\infty.$

First, it is known that the mild solution of (\ref{24/05/17(1)}) can be represented explicitly as
\begin{equation}
\label{02/12/1677(3)}
y(t, \phi) = G(t)\phi_0 + \int^0_{-r}G(t+\theta)S\phi_1(\theta)d\theta + \int^t_0 G(t-s)L(y(s, \phi), y_s(\phi))dW(s),
\end{equation}
where $G(\cdot)$ is the fundamental solution of (\ref{24/05/17(1)}) and $S\in {\mathscr L}(L^2_r)$ is the associated structure operator.  To proceed further, let us consider the difference of two mild solutions of (\ref{24/05/17(1)}) with distinct initial data $\phi,\,\psi\in {\cal H}$:
\[
\Sigma(t, \phi, \psi) =y(t, \phi) -y(t, \psi),\hskip 20pt t\ge 0.\]
Note that, by assumption, we have according to Proposition \ref{25/09/08(10)} that $\|G(t)\|\le Me^{-\gamma t}$, $M\ge 1$, $\gamma>0$ for all $t\ge 0$.
By virtue of (\ref{24/05/17(2)}), (\ref{02/12/1677(3)}) and It\^o isometry, it follows that
\begin{equation}
\label{04/12/16(1)}
\begin{split}
&{\mathbb E}\|\Sigma(t, \phi, \psi)\|^2_{\mathbb H}\\
 &\le 3{\mathbb E}\Big\{\|G(t)(\phi_0-\psi_0)\|^2_{\mathbb H} + \Big\|\int^0_{-r} G(t+\theta)S(\phi_1(\theta) -\psi_1(\theta))d\theta\Big\|^2_{\mathbb H}\\
&\,\,\,\,\,\, + \Big\|\int^t_0 G(t-s)\big[L(y(s, \phi), y_s(\phi))-L(y(s, \psi), y_s(\psi))\big]dW(s)\Big\|^2_{\mathbb H}\Big\}\\
&\le 3(M^2+M^2\|S\|^2_{{\mathscr L}(L^2_r)}re^{2\gamma r})e^{-2\gamma t}\|\phi-\psi\|^2_{\cal H}+ 3M^2\alpha_1\int^t_0 e^{-2\gamma(t-s)}{\mathbb E}\|y(s, \phi)-y(s, \psi)\|^2_{\mathbb H}ds\\
&\,\,\,\,\,+ 3M^2\alpha_2\int^t_0 e^{-2\gamma(t-s)}\int^0_{-r}{\mathbb E}\|y(s+\theta, \phi)-y(s+\theta, \psi)\|^2_{\mathbb H} \kappa(d\theta)ds,\hskip 20pt t\ge 0.
\end{split}
\end{equation}
On the other hand, it is easy to see that for any $t\ge 0$,
\begin{equation}
\label{25/05/17(1)}
\begin{split}
&\int^t_0 e^{-2\gamma(t-s)}\int^0_{-r} {\mathbb E}\|y(s+\theta, \phi)-y(s+\theta, \psi)\|^2_{\mathbb H} \kappa(d\theta)ds\\
 &\le \int^0_{-r} \int^t_{-r} e^{-2\gamma(t-s+\theta)}{\mathbb E}\|y(s, \phi)-y(s, \psi)\|^2_{\mathbb H}ds \kappa(d\theta)\\
&\le e^{2\gamma r}\kappa([-r, 0])\Big(\|\phi-\psi\|^2_{\cal H} + \int^t_0 e^{-2\gamma(t-s)}{\mathbb E}\|y(s, \phi)-y(s, \psi)\|^2_{\mathbb H}ds\Big).
\end{split}
\end{equation}
Substituting (\ref{25/05/17(1)}) into (\ref{04/12/16(1)}), one can have that for $t\ge 0$,
\begin{equation}
\label{26/05/17(1)}
\begin{split}
{\mathbb E}\|y(t, \phi)-y(t, \psi)\|^2_{\mathbb H}\le C_1e^{-2\gamma t}\|\phi-\psi\|^2_{\cal H} + C_2\int^t_0 e^{-2\gamma(t-s)}\|y(s, \phi)-y(s, \psi)\|^2_{\mathbb H} ds,
 \end{split}
 \end{equation}
 where
 \[
 C_1= 3M^2(1+\|S\|^2_{{\mathscr L}(L^2_r)}re^{2\gamma r}) + 3M^2\alpha_2 e^{2\gamma r}\kappa([-r, 0])>0\]
  and
  \[
  C_2 = 3M^2(\alpha_1+ \alpha_2 e^{2\gamma r}\kappa([-r, 0]))>0.\]
  Now, by multiplying $e^{2\gamma t}$ on both sides of (\ref{26/05/17(1)}), we obtain for any $t\ge 0$ that
\[
e^{2\gamma t}{\mathbb E}\|y(t, \phi)-y(t, \psi)\|^2_{\mathbb H} \le C_1\|\phi-\psi\|^2_{\cal H} + C_2\int^t_0 e^{2\gamma s}{\mathbb E}\|y(s, \phi)-y(s, \psi)\|^2_{\mathbb H} ds.\]
Hence, letting $\alpha = 2\gamma-3M^2(\alpha_1+ \alpha_2 e^{2\gamma r}\kappa([-r, 0]))>0$ and using the well-known Gronwall lemma, we have
\begin{equation}
\label{30/05/17(1)}
{\mathbb E}\|y(s, \phi)-y(s, \psi)\|^2_{\mathbb H}\le  C_1\|\phi-\psi\|^2_{\cal H}e^{-\alpha t},\hskip 20pt t\ge 0.
\end{equation}
Further, we may obtain that for all $t\ge 0$,
\begin{equation}
\label{05/15/16(1)}
\begin{split}
{\mathbb E}\int^0_{-r} \|y(t+\theta, \phi)-y(t+\theta, \psi)\|^2_{\mathbb H}d\theta &\le C_1\int^0_{-r} \|\phi-\psi\|^2_{\cal H}e^{-\alpha(t+\theta)}d\theta\\
&\le  rC_1\|\phi-\psi\|^2_{\cal H} e^{-\alpha (t-r)}.
\end{split}
\end{equation}
Hence, the relation (i) holds.

Next, we show that (ii) is also valid. To this end, first note that
\[
(a+b)^2\le (1+\varepsilon)a^2 + \Big(1+\frac{1}{\varepsilon}\Big)b^2\hskip 10pt \hbox{for any}\hskip 10pt a,\,\,b\ge 0\hskip 10pt \hbox{and}\hskip 10pt \varepsilon>0.\]
 Then by using (\ref{02/12/1677(3)}), (\ref{05/15/16(1)}) we can utilize It\^o's isometry and carry out a similar argument to (\ref{25/05/17(1)}) to have that for $t\ge 0$,
\begin{equation}
\label{05/12/16(10)}
\begin{split}
&{\mathbb E}\|y(t, \phi)\|^2_{\mathbb H}\\ &\le 3\Big\{\|G(t)\phi_0\|^2_{\mathbb H} + \Big\|\int^0_{-r} G(t-s)S\phi_1(\theta)d\theta\Big\|^2_{\mathbb H} + {\mathbb E}\int^t_0 \|G(t-s)L(y(s), y_s)\|^2_{{\mathscr L}_2(K_Q, {\mathbb H})}ds\Big\}\\
&\le 3M^2(1+\|S\|^2_{{\mathscr L}(L^2_r)}re^{2\gamma r})e^{-2\gamma t}\|\phi\|^2_{\cal H} + 3M^2\int^t_0 e^{-2\gamma(t-s)}\Big((1+\varepsilon)\alpha_1{\mathbb E}\|y(s, \phi)\|^2_{\mathbb H}\\
 &\,\,\,\,\, + (1+\varepsilon)\alpha_2\int^0_{-r}{\mathbb E}\|y(s+\theta, \phi)\|^2_{\mathbb H}\kappa(d\theta) + \Big(1+\frac{1}{\varepsilon}\Big)\|L(0, 0)\|^2_{{\mathscr L}_2(K_Q, {\mathbb H})}\Big)ds\\
&\le 3M^2(1 + \|S\|^2_{{\mathscr L}(L^2_r)}re^{2\gamma r} + e^{2\gamma r}\kappa([-r, 0])(1+\varepsilon)\alpha_2 e^{-2\gamma t}\|\phi\|^2_{\cal H}\\
 &\,\,\,\,\, + \frac{3M^2}{2\gamma}\Big(1+\frac{1}{\varepsilon}\Big)\|L(0, 0)\|^2_{{\mathscr L}_2(K_Q, {\mathbb H})} + 3M^2\alpha_1(1+\varepsilon)\int^t_0 e^{-2\gamma(t-s)}{\mathbb E}\|y(s, \phi)\|^2_{{\mathbb H}}ds \\
&\,\,\,\,\,+ 3M^2\alpha_2 e^{2\gamma r}(1+\varepsilon)\kappa([-r, 0])\int^t_0 e^{-2\gamma(t-s)}{\mathbb E}\|y(s, \phi)\|_{\mathbb H}^2ds.
\end{split}
\end{equation}
 By choosing $\varepsilon>0$ sufficiently small and using condition (\ref{26/05/17(10)}) and  the well-known Gronwall lemma, we thus obtain
\begin{equation}
\label{30/05/17(2)}
\sup_{t\ge 0}{\mathbb E}\|y(t, \phi)\|^2_{\mathbb H} \le C_3\|\phi\|^2_{\cal H} + \frac{3M^2}{2\gamma}\Big(1+\frac{1}{\varepsilon}\Big)\|L(0, 0)\|^2_{{\mathscr L}_2(K_Q, {\mathbb H})}<\infty,
\end{equation}
where
\[
C_3 =3M^2(1 + \|S\|^2_{{\mathscr L}(L^2_r)}re^{2\gamma r} + (\alpha_1+\alpha_2)(1+\varepsilon)e^{2\gamma r}\kappa([-r, 0]))>0.\]
By carrying out a similar argument to that of (\ref{05/15/16(1)}) and taking (\ref{05/12/16(10)}) into account, we easily have
\begin{equation}
\label{05/12/16(11)}
\sup_{t\ge 0}{\mathbb E}\int^0_{-r} \|y(t+\theta, \phi)\|^2_{\mathbb H}d\theta \le rC_3e^{\alpha r}\|\phi\|^2_{\cal H} + \frac{3M^2}{2\gamma}\Big(1+\frac{1}{\varepsilon}\Big)\|L(0, 0)\|^2_{{\mathscr L}_2(K_Q, {\mathbb  H})},
\end{equation}
and the claim (ii) is thus verified. Last, (\ref{30/05/17(2)}) and (\ref{05/12/16(11)}) together imply that
\begin{equation}
\label{05/12/16(12)}
\begin{split}
\sup_{t\ge 0}{\mathbb E}\|Y(t, \phi)\|^2_{\cal H} \le C_3(1+ re^{\alpha r})\|\phi\|^2_{\cal H} + \frac{3M^2}{2\gamma}\Big(1+\frac{1}{\varepsilon}\Big)\|L(0, 0)\|^2_{{\mathscr L}_2(K_Q, {\mathbb H})}<\infty.
\end{split}
\end{equation}
Employing the invariance of $\pi\in {\mathscr P}({\cal H})$ and integrating with respect to $\pi$ on both sides of (\ref{05/12/16(12)}) lead to
\begin{equation}
\label{05/12/16(50)}
\pi(\|\cdot\|_{\cal H})<\infty.
\end{equation}
Now we show the uniqueness of stationary distributions. If $\pi'\in {\mathscr P}({\cal H})$ is another stationary distribution for $Y(t) = (y(t), y_t)$, $t\ge 0$, of equation (\ref{24/05/17(1)}). Let $f\in C_{LB}({\cal H})$, the family of all bounded and Lipschitz continuous functions on ${\cal H}$. Then by virtue of (\ref{30/05/17(1)}), (\ref{05/15/16(1)}), (\ref{05/12/16(50)}), H\"older's inequality and the invariance of $\pi(\cdot),\,\pi'(\cdot)\in {\mathscr P}({\cal H})$, it follows that
\begin{equation}
\label{05/12/16(51)}
|\pi(f)-\pi'(f)|\le \int_{{\cal H}\times{\cal H}}|{\mathbb P}_tf(\phi)- {\mathbb P}_tf(\psi)|\pi(d\phi)\pi'(d\psi)\le C_4e^{-\alpha t},\hskip 15pt t\ge 0,
\end{equation}
for some constant $C_4>0$.
This implies the uniqueness of stationary distributions by letting $t\to \infty$ in (\ref{05/12/16(51)}). The proof is thus complete.
\end{proof}

\section{\large Systems Driven by L\'evy Jump Processes}

First, let $Z$ be a ${K}$-valued L\'evy process with its L\'evy triple $(0, Q, \nu)$.
For $t>0$ and $\Gamma\in {\mathscr B}({K}-\{0\}),$ we define a Poisson random measure generated by $Z(t)$ as
\[
N(t, \Gamma) = \sum_{s\in (0, t]}{\bf 1}_\Gamma(\Delta Z(s)),\]
where $\Delta Z(t) := Z(t)-Z(t-)$ for $t\ge 0$ and the compensated Poisson random measure is given by
\[
\tilde N(t, \Gamma) = N(t, \Gamma)-t\nu(\Gamma),\hskip 15pt t\ge 0,\hskip 15pt \Gamma\in {\mathscr B}({K}-\{0\}).\]
It is well known that the L\'evy process $Z(t)$ has the following L\'evy-It\^o decomposition
\begin{equation}
\label{30/10/17(4)}
Z(t) = W_Q(t) + \int_{{K}-\{0\}}z\tilde N(t, dz)
\end{equation}
with Tr$\,(Q)<\infty$.

Consider a second-order stochastic retarded differential equation driven by the L\'evy process $Z$ in $H$, i.e., for $t\ge 0$,
\begin{equation}
\label{12/09/17(1348967678)}
\begin{cases}
d\Big(\displaystyle\frac{du(t)}{dt}\Big)  + Au(t)dt = Bu'(t)dt+ Mu_tdt + Nu'_tdt  + R(u(t-), u'(t-), u_{t-}, u'_{t-})dZ(t),\\
u(0)= \phi_{0, 1}\in V,\,\,\,u'(0)=\phi_{0, 2}\in H,\\
u_0= \phi_{1,1}\in L^2([-r, 0], V),\,\,\, u'_0=\phi_{1, 2}\in L^2([-r, 0], H),
\end{cases}
\end{equation}
where $u(t-) = \lim_{s\uparrow t}u(s)$, $u_{t-}(\theta) := \lim_{s\uparrow t}u(s+\theta)$, for $t\ge 0$ and $\theta\in [-r, 0]$, and the non-linear mapping
\[
R: V\times H\times L^2([-r, 0], V)\times L^2([-r, 0], H)\to {\mathscr L}(K, H)\]
 is assumed to be Borel measurable and there exist constants $\alpha_1,\,\alpha_2>0$ and a finite measure $\kappa(\cdot)$ on $[-r, 0]$
 such that
\begin{equation}
\label{08/10/17(26023)}
\begin{split}
\|R(\phi)-R(\psi)&\|_{{\mathscr L}(K, H)}^2\le \alpha_1\Big(\|\phi_{0, 1}-\psi_{0, 1}\|^2_V + \|\phi_{0, 2}-\psi_{0, 2}\|^2_H\Big)\\
&\,\, + \alpha_2\Big(\int^0_{-r} \|\phi_{1, 1}(\theta) - \psi_{1,1}(\theta)\|^2_V \kappa(d\theta) + \int^0_{-r} \|\phi_{1, 2}(\theta) - \psi_{1,2}(\theta)\|^2_H \kappa(d\theta)\Big)
\end{split}
\end{equation}
for any $\phi = (\phi_0, \phi_1) = (\phi_{0, 1}, \phi_{0, 2}, \phi_{1,1}, \phi_{1, 2}),\, \psi =(\psi_0, \psi_1) =  (\psi_{0, 1}, \psi_{0, 2}, \psi_{1,1}, \psi_{1, 2})\in V\times H\times L^2([-r, 0], V)\times L^2([-r, 0], H)$.
Defining $y(t) = \displaystyle{u(t)\choose u'(t)}$,  $y_t = \displaystyle{u_t\choose u'_t}$ and a nonlinear mapping $L$,
\[
L(\phi_0, \phi_1) = {0\choose R(\phi_{0, 1}, \phi_{0, 2}, \phi_{1,1}, \phi_{1,2})}\]
for any $\phi = (\phi_0, \phi_1) = (\phi_{0, 1}, \phi_{0, 2}, \phi_{1,1}, \phi_{1, 2})\in V\times H\times L^2([-r, 0], V)\times L^2([-r, 0], H)$,
 we can transform (\ref{12/09/17(1348967678)}) into a first-order stochastic system in ${\mathbb H}$
\begin{equation}
\label{24/05/17(1503)}
\begin{cases}
dy(t) = \Lambda y(t)dt + Fy_tdt + L(y(t-), y_{t-})dZ(t),\,\,\,\,t\ge 0,\\
y(0)= \phi_0 = \displaystyle{\phi_{0, 1}\choose \phi_{0, 2}}\in {\mathbb H},\,\,\,y_0 =\phi_1 = {\phi_{1, 1}\choose  \phi_{1, 2}}\in L^2([-r, 0], {\mathbb H}),
\end{cases}
\end{equation}
where $L: {\cal H}\to {\mathscr L}(K, {\mathbb H})$ is clearly Borel measurable and  (\ref{08/10/17(26023)}) implies immediately that
 \begin{equation}
\label{24/05/17(269)}
\begin{split}
\|L(\phi)-L(\psi)\|^2_{{\mathscr L}(K, {\mathbb H})}\le \alpha_1\|\phi_0-\psi_0\|^2_{\mathbb H} + &\,\alpha_2\int^0_{-r} \|\phi_1(\theta)-\psi_1(\theta)\|^2_{L^2([-r, 0], {\mathbb H})}\kappa(d\theta),\\
&\forall\,\phi=(\phi_0, \phi_1),\,\psi=(\psi_0, \psi_1)\in {\cal H}.
\end{split}
\end{equation}

 \begin{theorem}
 \label{05/11/17(10)}
Assume that $\|e^{t{\cal A}}\|\le Me^{-\gamma t}$, $M\ge 1$, $\gamma>0$ for all $t\ge 0$. Suppose that (\ref{08/10/17(26023)}) holds  and
\begin{equation}
\label{26/05/17(120)}
\int_{z\not= 0}\|z\|^2_{K}\nu(dz)<\infty,
\end{equation}
and further
\begin{equation}
\label{26/05/17(101)}
2\gamma > 3M^2\Big(Tr(Q) + \int_{z\not= 0}\|z\|^2_K \nu(dz)\Big)(\alpha_1+\alpha_2e^{2\gamma r}\kappa([-r, 0])),
\end{equation}
then
there exists a unique stationary distribution $\mu \in {\mathscr P}({\cal H})$ for $Y(t)=(y(t), y_t)$, $t\ge 0$, of (\ref{12/09/17(1348967678)}).
\end{theorem}
\begin{proof}
By the variation-of-constants formula, one has
\begin{equation}
\label{30/10/17(3)}
y(t, \phi) = G(t)\phi_0 + \int^0_{-r} G(t+\theta)S\phi_1(\theta)d\theta + \int^t_0 G(t-s)L(y(s-), y_{s-})dZ(s).
\end{equation}
Thus, by virtue of (\ref{30/10/17(4)}), we have
\begin{equation}
\label{30/10/17(5)}
\begin{split}
y(t, \phi) &= G(t)\phi_0 + \int^0_{-r} G(t+\theta)S\phi_1(\theta)d\theta + \int^t_0 G(t-s)L(y(s), y_{s})dW_Q(s)\\
&\,\,\,\,\, + \int^t_0 \int_{K-\{0\}}G(t-s)L(y(s-), y_{s-})z\tilde N(ds, dz)
\end{split}
\end{equation}
Carrying out a similar argument to that of (\ref{05/12/16(10)}) and taking (\ref{26/05/17(120)}) into account, one can easily obtain that
\[
\begin{split}
{\mathbb E}\|y(t, \phi)\|^2_{\mathbb H} &\le 6\Big\{\|G(t)\phi_0\|^2_{\mathbb H} + \Big\|\int^0_{-r} G(t-s)S\phi_1(\theta)\Big\|^2_{\mathbb H}d\theta\Big\}\\
&\,\,\,\,\ + 3{\mathbb E}\Big\|\int^t_0 G(t-s)L(y(s), y_{s})dW_Q(s)\Big\|^2_{\mathbb H}\\
 &\,\,\,\,\,+ 3{\mathbb E}\Big\|\int^t_0 \int_{z\not= 0}G(t-s)L(y(s-), y_{s-})z\tilde N(ds, dz)\Big\|^2_{\mathbb H}\\
& := I_1(t) +I_2(t) + I_3(t).
\end{split}
\]
Hence, we have by using (\ref{24/05/17(269)}) that for any $t\ge 0$,
\[
I_1(t)\le 6M^2(1+ \|S\|^2_{{\mathscr L}(L^2_r)}re^{2\gamma r})e^{-2\gamma t}\|\phi\|^2_{\cal H},\]
\[
\begin{split}
I_2(t) \le &\,\, 3M^2Tr(Q)\int^t_0 e^{-2\gamma(t-s)}\|L(y(s), y_{s})\|^2_{{\mathscr L}(K, {\mathbb H})}ds\\
 \le &\,\, 6M^2Tr(Q)\alpha_1\int^t_0 e^{-2\gamma(t-s)}{\mathbb E}\|y(s)\|^2_{\mathbb H}ds\\
  &\,\,+ 6M^2Tr(Q)\alpha_2 \int^t_0 e^{-2\gamma(t-s)}\int^0_{-r} {\mathbb E}\|y(s+\theta)\|^2_{\mathbb H}\kappa(d\theta)ds\\
  &\,\, + 6M^2Tr(Q)\int^t_0 e^{-2\gamma(t-s)}\|L(0, 0)\|^2_{{\mathscr L}(K, {\mathbb H})}ds,
 \end{split}\]
and
\[
\begin{split}
I_3(t) &\le 3M^2\int^t_0 \int_{z\not= 0}e^{-2\gamma (t-s)}\|L(y(s-), y_{s-})\|^2_{{\mathscr L}(K, {\mathbb H})}\|z\|^2_H N(ds, dz)\\
&\le 6M^2\int_{z\not= 0}\|z\|^2_H \nu (dz)\Big(\alpha_1\int^t_0 e^{-2\gamma(t-s)}{\mathbb E}\|y(s)\|^2_{\mathbb H}ds\\
&\,\,\,\,\, + \alpha_2 \int^t_0 e^{-2\gamma(t-s)}\int^0_{-r} {\mathbb E}\|y(s+\theta)\|^2_{\mathbb H}\kappa(d\theta)ds + \int^t_0 e^{-2\gamma(t-s)}\|L(0, 0)\|^2_{{\mathscr L}(K, {\mathbb H})}ds\Big).
\end{split}
\]
Therefore, under the conditions (\ref{26/05/17(120)}) and (\ref{26/05/17(101)}),
 we may obtain similarly to (\ref{05/12/16(10)})-(\ref{05/12/16(12)}) the following relation
\begin{equation}
\label{30/10/17(6)}
\sup_{t\ge 0}{\mathbb E}\|y(t, \phi)\|^2_{\mathbb H}<\infty,
\end{equation}
which shows the conclusion (ii) in Lemma \ref{25/10/17(2)}.
Finally, the desired assertion (ii) in Lemma \ref{25/10/17(2)} can be concluded by imitating the arguments of Theorem \ref{05/06/17(80)}.
\end{proof}

It is well-known that for a L\'evy jump process, there is, in general, no finite second moment (\ref{26/05/17(120)}). As a result, there is generally no It\^o's isometry as shown in Theorem  \ref{05/11/17(10)}. To go around this difficulty, let us formulate our system in a product Banach space ${\cal H}_1 = {\mathbb H}\times L^1([-r, 0], {\mathbb H})$, equipped with the norm
\[
\|\phi\|_{{\cal H}_1} = \|\phi_0\|_{\mathbb H} + \int^0_{-r} \|\phi_1(\theta)\|_{\mathbb H}d\theta,\hskip 20pt \phi=(\phi_0, \phi_1)\in {\cal H}_1.\]
Consider a linear Cauchy problem with memory in the Hilbert space $H$,
\begin{equation}
\label{12/09/17(13489676)}
\begin{cases}
d\Big(\displaystyle\frac{du(t)}{dt}\Big)  + Au(t)dt = Bu'(t)dt+ Mu_tdt + Nu'_tdt  + RdZ(t),\hskip 15pt t\ge 0,\\
u(0)= \phi_{0, 1}\in V,\,\,\,u'(0)=\phi_{0, 2}\in H,\\
u_0= \phi_{1,1}\in L^2([-r, 0], V),\,\,\, u'_0=\phi_{1, 2}\in L^2([-r, 0], H),
\end{cases}
\end{equation}
where $A$, $B$, $M$, $N$ and $V$ are given as in Section 4, $Z$ is a ${K}$-valued L\'evy process with its L\'evy triple $(0, 0, \nu)$ and $R\in {\mathscr L}(K, H)$.
Let us rewrite this problem into a first-order stochastic delay equation.
Precisely, by defining $y(t) = \displaystyle{u(t)\choose u'(t)}$, so $y_t = \displaystyle{u_t\choose u'_t}$ and a bounded linear operator $L\in {\mathscr L}(K, {\mathbb H})$ by
\[
L = {0\choose R}: K\to {\mathbb H},\,\,\,Lx={0\choose Rx},\,\,x\in K,\]
 we can transform (\ref{12/09/17(13489676)}) into a first-order stochastic system
\begin{equation}
\label{24/05/17(145)}
\begin{cases}
dy(t) = \Lambda y(t)dt + Fy_tdt + LdZ(t),\,\,\,\,t\ge 0,\\
y(0)=\phi_0,\,\,y_0 =\phi_1,\,\,\phi=(\phi_0, \phi_1)\in {\cal H}_1.
\end{cases}
\end{equation}
Note that $\|L\|=\|R\|$. Indeed, for any $x\in K$,
\[
\|Lx\|^2_{\mathbb H} = \|Rx\|^2_H\le \|R\|^2\|x\|^2_K,\]
i.e., $\|L\|\le \|R\|.$ On the other hand, by definition, it is true that
\[
\|Rx\|^2_H = \|Lx\|^2_{{\mathbb H}}\le \|L\|^2\|x\|^2_K,\hskip 15pt \forall\, x\in K,\]
that is, $\|R\|\le \|L\|$. Hence, $\|R\|=\|L\|.$

\begin{theorem}
\label{05/06/17(20)}
 Suppose that the Green operator $G(t)$ is exponentially stable, i.e.,  $\|G(t)\|\le Me^{-\gamma t}$, $M\ge 1$, $\gamma>0$ for all $t\ge 0$. Assume further that
\[
\int_{\|z\|_{K}>1} \|z\|_{K}\nu(dz)<\infty.\]
then
there exists a unique stationary distribution $\mu \in {\mathscr P}({\cal H})$ for $Y(t)=(y(t), y_t)$, $t\ge 0$, of (\ref{24/05/17(145)}).
\end{theorem}
\begin{proof}
We need only verify the following assertions: for any bounded subset $U$ of ${\cal H}_1$,

(i) $\lim_{t\to\infty} \sup_{\phi,\,\psi\in U}{\mathbb E}\|Y(t, \phi)-Y(t, \psi)\|_{{\cal H}_1}=0;$

(ii) $\sup_{t\ge 0}\sup_{\phi\in U}{\mathbb E}\|Y(t, \phi)\|_{{\cal H}_1}<\infty.$

\noindent By the well-known L\'evy-It\^o decomposition theorem, one can get
\[
Z(t) =\int_{\|z\|_{K}\le 1} z\tilde N(t, z) + \int_{\|z\|_{K}>1} zN(t, dz),\hskip 20pt t\ge 0.\]
Hence, according to the variation-of-constants formula the mild solution of (\ref{24/05/17(145)}) is given by
\begin{equation}
\label{31/05/17(3)}
\begin{split}
y(t, \phi) = &\,\, G(t)\phi_0 + \int^0_{-r}G(t+\theta)S\phi_1(\theta)d\theta + \int^t_0 \int_{\|z\|_{{K}}\le 1}G(t-s)Lz\tilde N(ds, dz)\\
 &\,\,+ \int^t_0 \int_{\|z\|_{{K}}> 1}G(t-s)Lz N(ds, dz)\\
 =:&\,\, \sum^4_{j=1}I_j(t),
 \end{split}
\end{equation}
By assumption, it is immediate to see that
\begin{equation}
\label{30/10/17(1)}
\sup_{t\ge 0}{\mathbb E}(\|I_1(t)\|_{\mathbb H} + \|I_2(t)\|_{\mathbb H}) \le M(1+r\|S\|_{{\mathscr L}(L^2_r)})\|\phi\|_{{\cal H}_1}<\infty.
\end{equation}
Note from the H\"older inequality, It\^o's isometry and the uniform boundedness of $G(t)$, $t\ge 0$, that
\begin{equation}
\label{31/05/17(4)}
\begin{split}
\sup_{t\ge 0} {\mathbb E}\|I_3(t)\|_{\mathbb H} &\le \Big(\sup_{t\ge 0}\int^t_0 \|G(t-s)L\|^2ds\int_{\|z\|_{K}\le 1}\|z\|^2_{K}\nu(dz)\Big)^{1/2}\\
&\le \Big(\frac{M^2\|R\|^2}{2\gamma}\int_{\|z\|_{K}\le 1}\|z\|^2_{K}\nu(dz)\Big)^{1/2}<\infty
\end{split}
\end{equation}
since $\nu(\cdot)$ is a L\'evy measure. On the other hand, by assumption it follows that
\begin{equation}
\label{31/05/17(5)}
\begin{split}
\sup_{t\ge 0}{\mathbb E}\|I_4(t)\|_{\mathbb H} &\le \int_{\|z\|_{K}>1}\|z\|_{K}\nu(dz)\sup_{t\ge 0}\Big(\int^t_0 \|G(t-s)L\|ds\Big)\\
&= \frac{M\|R\|}{\gamma}\int_{\|z\|_{K}>1}\|z\|_{K}\nu(dz)<\infty.
\end{split}
\end{equation}
Hence, (\ref{31/05/17(3)})-(\ref{31/05/17(5)}) yield the relation
\begin{equation}
\label{31/05/17(6)}
\sup_{t\ge 0}{\mathbb E}\|y(t, \phi)\|_{{\mathbb H}} = M(1+r\|S\|_{{\mathscr L}(L^2_r)})\|\phi\|_{{\cal H}_1} + C<\infty\hskip 10pt\hbox{for some}\hskip 10pt  C>0.
\end{equation}
From (\ref{31/05/17(6)}), it is easy to have that
\[
\sup_{t\ge 0}\int^0_{-r} {\mathbb E}\|y(t+\theta, \phi)\|_{\mathbb H}d\theta \le r(M(1+r\|S\|_{{\mathscr L}(L^2_r)})\|\phi\|_{{\cal H}_1} + C)<\infty,\]
which, in addition to (\ref{31/05/17(6)}), immediately implies that
\[
\sup_{t\ge 0}\sup_{\phi\in U}{\mathbb E}\|Y(t, \phi)\|_{{\cal H}_1}<\infty.\]
On the other hand, for $\phi,\,\psi\in U$, it follows from (\ref{31/05/17(3)}) that
\[
y(t, \phi)-y(t, \psi) =G(t)(\phi_0 -\psi_0) + \int^0_{-r} G(t+\theta)S(\phi_1(\theta)-\psi_1(\theta))d\theta,\]
which implies that
\begin{equation}
\label{05/06/17(1)}
\begin{split}
\sup_{\phi,\,\psi\in U}{\mathbb E}&\|y(t, \phi)-y(t, \psi)\|_{\mathbb H}\\
&\le \sup_{\phi,\,\psi\in U}\Big(M\|\phi_0-\psi_0\|_{\mathbb H} e^{-\gamma t} + Me^{-\gamma t}r\|S\|_{{\mathscr L}(L^2_r)}\int^0_{-r} \|\phi_1(\theta)-\psi_1(\theta)\|_{\mathbb H}d\theta\Big)\\
&\le \sup_{\phi,\,\psi\in U}M(1+r\|S\|_{{\mathscr L}(L^2_r)})\|\phi-\psi\|_{{\cal H}_1}e^{-\gamma t}\\
&\to 0\,\,\,\,\,\,\hbox{as}\,\,\,t\to\infty.
\end{split}
\end{equation}
This further implies that
\begin{equation}
\label{05/06/17(2)}
\begin{split}
\sup_{\phi,\,\psi\in U}&{\mathbb E}\|Y(t, \phi)-Y(t, \psi)\|_{{\cal H}_1}\\
 &\le \sup_{\phi,\,\psi\in U}{\mathbb E}\|y(t, \phi)-y(t, \psi)\|_{\mathbb H} + \sup_{\phi,\,\psi\in U}\int^0_{-r}{\mathbb E}\|y(t+\theta, \phi)-y(t+\theta, \psi)\|_{\mathbb H}d\theta\\
 &\le \sup_{\phi,\,\psi\in U}M(1+r\|S\|_{{\mathscr L}(L^2_r)})(1+ e^{\gamma r})e^{-\gamma t}\\
&\to 0\,\,\,\,\,\,\hbox{as}\,\,\,t\to\infty.
\end{split}
\end{equation}
Hence, there exists a stationary solution $\mu(\cdot)$ of $Y(t, \phi)$ for (\ref{24/05/17(145)}). On the other hand, if $\tilde\mu(\cdot)\in {\mathscr P}({\cal H}_1)$ is also a stationary solution, then for any $f\in C_{LB}({\cal H}_1)$, by the invariance of $\mu(\cdot)$, $\tilde\mu(\cdot)\in {\mathscr P}({\cal H}_1)$, it follows from (\ref{05/06/17(2)}) that
\begin{equation}
\label{05/06/17(10)}
|\mu(f) -\tilde\mu(f)|\le \int_{{\cal H}_1}\int_{{\cal H}_1}|{\mathbb P}_tf(\phi) - {\mathbb P}_tf(\psi)|\mu(d\phi)\tilde\mu(d\psi)\le Ce^{-\gamma t},\hskip 15pt t\ge 0,
\end{equation}
for some $C>0$. This  implies the uniqueness of stationary distributions by taking $t\to\infty$ in (\ref{05/06/17(10)}).
The proof is complete now.
\end{proof}

Examining the proof of Theorem \ref{05/06/17(20)}, the technique employed therein applies to (\ref{24/05/17(145)}) with the L\'evy triple $(0, 0, \nu)$ and a retarded SDE (\ref{12/09/17(13489676)})
with an uniformly bounded diffusion coefficient $R$, i.e., $\sup_{x\in H}\|R(x)\|_{{\mathscr L}(K, H)}< \infty$.

\section{\large Example}

In this section, we shall consider an example to illustrate our theory in the previous sections. To this end, we assume that there exist $\eta: [-r, 0]\to {\mathscr L}(V, H)$ and $\eta: [-r, 0]\to {\mathscr L}(H)$ of bounded variation such that
\[
M(\varphi) = \int^0_{-r} d\eta(\theta)\varphi(\theta)\hskip 15pt \forall\,\varphi\in W^{1, 2}([-r, 0], V),\]
and
\[
N(\varphi) = \int^0_{-r} d\zeta(\theta)\varphi(\theta)\hskip 15pt \forall\,\varphi\in W^{1, 2}([-r, 0], H).\]
Defining $\varrho: W^{1, 2}([-r, 0], V)\times W^{1,2}([-r, 0], H)\to {\mathbb H}$ by $\varrho := \left(\begin{array}{cc}
0&0\\
\eta & \zeta
\end{array}\right)$, then we obtain that $F$ is of the form
\[
F\left(\begin{array}{c}
\varphi_1\\
\varphi_1\end{array}\right) =  \int^0_{-r} d\varrho(\theta)\left(\begin{array}{c}
\varphi_1(\theta)\\
\varphi_1(\theta)
\end{array}\right).\]
Moreover, we can estimate for $a<0$,
\begin{equation}
\label{-7/11/17(10)}
\|F_{a+ib}y\|\le \Big\|\int^0_{-r} d\varrho(\theta) e^{(a+ib)\theta} y\Big\| \le Var(\varrho)^{0}_{-r}e^{-ar}\|y\|\hskip 20pt \forall\, y\in {\mathbb H}.
\end{equation}
For simplicity, let $r=1$ in the sequel.
Hence, Corollary \ref{06/11/17(2)} and Corollary \ref{05/11/17(50)}
yield the following result.
\begin{corollary}
Under the assumptions of Corollary \ref{05/11/17(50)}, the energy $\|y(t)\|^2_{\mathbb H} = \|u(t)\|^2_V + \|u'(t)\|^2_H$, $t\ge 0$, of the second-order abstract Cauchy problem (\ref{12/09/17(134896789)}) with delay decays exponentially if
\[
\sup_{b\in {\mathbb R}}\Big\{\|\eta_{ib}\|_{{\mathscr L}(V, H)} + \|\zeta_{ib}\|_{{\mathscr L}(H)}\Big\} \le \frac{\alpha\kappa}{2\alpha(3+\gamma) +\kappa}\]
where $\alpha$, $\gamma$ and $\kappa$ are given as in Lemma \ref{27/09/17(10)}
 and
Corollary \ref{05/11/17(50)}.
\end{corollary}

 Now let us consider a stochastic damped wave equation with delay
\begin{equation}
\label{20/09/17(341)}
\begin{cases}
d\Big(\displaystyle\frac{\partial u(t, \xi)}{\partial t} + 2\alpha u(t, \xi)\Big) = \frac{\partial^2}{\partial\xi^2}u(t, \xi)dt + c_1\frac{\partial}{\partial \xi}u(t-1, \xi)dt + c_2\frac{\partial}{\partial t}u(t-1, \xi)dt,\\
 \hskip 150pt  +\, \displaystyle\frac{\beta u(t-1, \xi)}{1+|u(t, \xi)|}dw(t),\hskip 10pt (t, \xi)\in {\mathbb R}_+\times (0, 1),\\
u(\theta, \xi)=\phi_{1,1}(\theta, \xi),\,\,\displaystyle\frac{\partial}{\partial\xi}u(\theta, \xi)=\phi_{1, 2}(\theta, \xi),\,\,(\theta, \xi)\in [-1, 0]\times (0, 1),\\
u(t, 0)= u(t, 1)=0,\,\,t\ge 0,
\end{cases}
\end{equation}
where $\alpha>0$, $\beta,\, c_1,\,c_2\in {\mathbb R}$, $c_1\not=0$ or $c_2\not= 0$ and $w$ is a standard one-dimensional Brownian motion. Here we assume that $\phi_{1, 1}(0, \cdot)\in H^1_0(0, 1)$ with the mapping $\theta\to \phi_{1, 1}(\theta, \cdot)\in H^1_0(0, 1)$ belongs to $L^2([-1, 0], H^1_0(0, 1))$, and $\phi_{1, 2}(0, \cdot)\in L^2(0, 1)$ with the mapping $\theta\to \phi_{1, 2}(\theta, \cdot)\in L^2(0, 1)$ belongs to $L^2([-1, 0], L^2(0, 1))$.

To write this problem in an abstract form, consider the space $V= H^1_0(0, 1)$, $H= L^2(0, 1)$ and ${\mathbb H} := V \times H.$ We define the operators
\[
\begin{split}
&A = -\Delta,\,\,\,{\mathscr D}(A) := \{u\in H^1_0:\, \Delta u \in L^2(0, 1)\},\\
&A^{1/2} = \sqrt{-\Delta},\,\,\, {\mathscr D}(A^{1/2}) = H^1_0(0, 1),\\
&Bu = 2\alpha u,\, {\mathscr D}(B) = H,\\
&\eta = c_1 \frac{\partial}{\partial{\xi}}\delta_{-1},\hskip 20pt \zeta = c_2 I\delta_{-1},
\end{split}
\]
where $\delta_{-1}$ is the point evaluation in $-1$. We note that $\|A^{1/2}u\|_{H} = \||\nabla u|\|_H$ for $u\in V$ and $\{u\in H^1_0(0, 1): \, \Delta u\in H\}=H^2_0(0, 1).$

Now let $\beta=0$ and by a direct calculation, it is not difficult to see that
\[
\sup_{b\in {\mathbb R}}\|\eta_{ib}\|_{{\mathscr L}(V, H)}\le |c_1|\hskip 15pt \hbox{and}\hskip 15pt \sup_{b\in {\mathbb R}}\|\zeta_{ib}\|_{{\mathscr L}(H)}\le |c_2|.\]
In Corollary \ref{05/11/17(20)}, where we can consider a damped wave equation without delay, we choose
\[
\Big(\|A^{1/2}BA^{-1/2}\| + 2\|A^{-1/2}\|\Big)^{-1} = \frac{\pi}{4\alpha + 2}.\]
Then we have that the solution in this case decays exponentially if
\[
|c_1|+|c_2| < \frac{\alpha\pi}{36\alpha + \pi}.\]
Now let us consider the stochastic delay partial differential equation (\ref{20/09/17(341)}) with $\beta\not= 0$. Define
\[
\gamma = \ln\frac{\alpha\pi}{(|c_1|+|c_2|)(36\alpha +\pi)}>0,\]
then we have by virtue of (\ref{-7/11/17(10)}) and Theorem \ref{05/06/17(80)}
that whenever
\[
0<|\beta|< \frac{2}{3}\gamma e^{-2\gamma},\]
 stochastic system (\ref{20/09/17(341)}) has a unique stationary process.

\vskip 35pt
\leftline{ {\bf\large References}}
\bibliographystyle{elsarticle-num}
\bibliography{<your-bib-database>}

\smallskip

\begin{thebibliography}{99}
\smallskip


\bibitem{jhbgycgy16} Bao, J. H., Yin, G. and Yuan, C. G. Stationary distributions for retarded stochastic differential equations without dissipativity. {\it Stochastics.} {\bf 89}, (2017), 530--549.


\bibitem{absp05} B\'atkai, A. and Piazzera, S. {\it Semigroups for Delay Equations}. Research Notes in Math., A.K. Peters, Wellesley, Massachusetts, (2005).







 \bibitem{dbgkkes84} Di Blasio, G., Kunisch, K. and Sinestrari, E. $L^2$-regularity for parabolic partial integrodifferential equations with delay in the highest-order derivatives. {\it J. Math. Anal. Appl.} {\bf 102} (1984), 38--57.

     \bibitem{rmd2003}  Dudley, R. M.
  {\it Real Analysis and Probability.} Second Edition, Cambridge University Press, (2003).



\bibitem{kjern00}  Engel, K-J. and Nagel, R.
  {\it One-Parameter Semigroups for Linear Evolution Equations.} Graduate Texts in Math. {\bf 194}, New York, Springer-Verlag, (2000).

  \bibitem{faho85} Fattorini, H.
  {\it Second Order Linear Differential Equations in Banach Spaces.} North Holland Math. Studies Series, {\bf 108}, (1985).

   \bibitem{flhjg2014} Liang, F. and  Gao, H. J. Stochastic nonlinear wave equation with memory driven by compensated Poisson random measures. {\it J. Math. Phys.} {\bf 55} (2014), 033503: 1--23.


   \bibitem{flzhg2014} Liang, F. and  Guo, Z. H. Asymptotic behavior for second order stochastic evolution equations with memory. {\it J. Math. Anal. Appl.} {\bf 419} (2014), 1333--1350.

   \bibitem{kliu06} Liu, K. {\it Stability of Infinite Dimensional Stochastic Differential Equations with Applications}.  Chapman \& Hall/CRC, London, New York, (2006).

        \bibitem{kliu08(0)} Liu, K. Stochastic retarded evolution equations: Green operators, convolutions and solutions. {\it Stoch. Anal. Appl.} {\bf 26} (2008), 624--650.

  \bibitem{kliu11} Liu, K. A criterion for stationary solutions of retarded  linear equations with additive noise. {\it Stoch. Anal. Appl.} {\bf 29} (2011), 799--823.

      \bibitem{sn1988} Nakagiri, S.  Structural properties of functional differential equations in Banach spaces. {\it Osaka J. Math.} {\bf 25},  (1988), 353--398.

  \bibitem{partha67}  Parthasarathy, K. P.
 {\it Probability Measures on Metric Spaces.} Academic Press, New York,
(1967).

   \bibitem{yrds09} Ren, Y. and Sun, D. D. Second-order neutral impulsive stochastic evolution equations with delay. {\it J. Math. Phys.} {\bf 50} (2009), 102709: 1-12.

 \bibitem{rysr12} Ren, Y. and Sakthivel, R. Existence, uniqueness and stability of mild solutions for second-order neutral stochastic evolution equations with infinite delay and Poisson jumps. {\it J. Math. Phys.} {\bf 53} (2012), 073517: 1-14.

\bibitem{atdl80} Taylor, A. and Lay, D.
  {\it Introduction to Functional Analysis.} Second Edition. John Wiley \& Sons, (1980).







\end{thebibliography}



\end{document}